\documentclass[11pt,reqno]{amsart}
\usepackage{amssymb,amsfonts,amsthm,amscd,stmaryrd,dsfont,esint,upgreek,constants,todonotes}
\usepackage[mathcal]{euscript}
\usepackage[latin1]{inputenc}   
\usepackage[mathcal]{euscript}
\usepackage{graphicx,color}
\numberwithin{equation}{section}
\newcommand{\N}{\mathbb N}

\newcommand{\R}{\mathbb R}

\newconstantfamily{C}{symbol=C}
\newconstantfamily{c}{symbol=c}
\newconstantfamily{O}{symbol=\Omega}
\newcommand{\vs}{\vskip.075in}

\def\XXint#1#2#3{{\setbox0=\hbox{$#1{#2#3}{\int}$}
\vcenter{\hbox{$#2#3$}}\kern-.5\wd0}}

\numberwithin{equation}{section}
\newtheorem{thm}{Theorem}[section]
\newtheorem{lem}[thm]{Lemma}
\newtheorem{cor}[thm]{Corollary}
\newtheorem{prop}[thm]{Proposition}

\theoremstyle{definition}
\newtheorem{defn}[thm]{Definition}
\newtheorem{rmk}[thm]{Remark}

\def\smallnegint{\mathop{\int\mkern-13mu
        \raise.5ex\hbox{${\scriptscriptstyle\diagup}$}}\nolimits}

\def\ep{\varepsilon}

\def\ssetminus{\,\raise.4ex\hbox{$\scriptstyle\setminus$}\,}

\newcommand{\be}{\begin{equation}}
\newcommand{\ee}{\end{equation}}

\renewcommand{\bar}{\overline}
\renewcommand{\tilde}{\widetilde}
\renewcommand{\hat}{\widehat}

\begin{document}

\thanks{P.C. was partially supported by P.S's Air Force Office for Scientific Research grant FA9550-18-1-0494 and by the Agence Nationale de la Recherche (ANR), project ANR-22-CE40-0010 COSS. P.E.S. was partially  supported by the National Science Foundation grant DMS-1900599, the Office for Naval Research grant N000141712095 and the Air Force Office for Scientific Research grant FA9550-18-1-0494.}

\title{An optimal control  problem of traffic flow on a junction}

\author[P. Cardaliaguet and P. E. Souganidis]{Pierre Cardaliaguet and Panagiotis E. Souganidis} 
\address{Universit\'e Paris-Dauphine, PSL Research University, Ceremade, 
Place du Mar\'echal de Lattre de Tassigny, 75775 Paris cedex 16 - France}
\email{cardaliaguet@ceremade.dauphine.fr }
\address{Department of Mathematics, University of Chicago, Chicago, Illinois 60637, USA}
\email{souganidis@uchicago.edu}

\dedicatory{Version: \today}

\begin{abstract} We investigate how to control optimally a traffic flow  through a junction on the line by acting only on speed reduction or traffic light at the junction. We show the existence of an optimal control and, under structure assumptions, provide optimality conditions. We use this analysis to investigate thoroughly  the maximization of the flux on a space-time subset and show  the existence of an optimal control which is bang-bang. 
\end{abstract}

\maketitle  

\tableofcontents

\section{Introduction}

In this note,  we investigate how to control optimally a traffic flow  through a junction on the line by acting only on speed reduction or traffic light at the junction. 
\vs
In our setting, the evolution of the density $\rho^A$ of the vehicles  is given by a conservation law of the form 
\be\label{eq.CL}
\left\{\begin{array}{ll}
\rho^A_t +(f^L(\rho^A))_x= 0 & \;  \text{in}\; (-\infty,0)\times (0,T),\\[1mm]
\rho^A_t +(f^R(\rho^A))_x= 0 & \;  \text{in}\; (0, \infty)\times (0,T),\\[1mm]
\rho^A_t(0,t) \in \mathcal G(A)(t)  &\; \text{a.e. in} \; (0,T), \\[1mm]
\rho^A(x,0)= \rho_0(x) & \;\text{in} \; \R.
\end{array}\right.
\ee
Here $f^L:[0,R^L]\to \R$ and $f^R:[0, R^R]\to \R$ are two strictly concave flux functions vanishing respectively at $0$ and $R^L$ and at $0$ and $R^R$. The measurable initial condition $\rho_0:\R\to \R$ is such that $\rho_0\in [0,R^L]$ on $(-\infty,0)$ and $\rho_0\in [0,R^R]$ on $(0,\infty)$. The time-dependent control is the flux limiter $A=A(t)$ which acts on the traffic through the junction condition, that is 
$$
\rho^A_t(0,t)=(\rho^A_t(0^-,t),\rho^A_t(0^+,t)) \in \mathcal G(A)(t) \qquad \text{a.e. in} \ (0,T),
$$
where $\rho^A_t(0^-,t)$ and $\rho^A_t(0^+,t)$ are understood in the sense of traces, which exist since $\rho^A$ solves a scalar conservation law with a strictly concave flux. Finally,  the time dependent germ $\mathcal G(A)(t)$ is given by 
$$
\mathcal G(A)(t)= \left\{ e=(e^L,e^R)\in Q, \; f^L(e^L)=f^R(e^R) = \min\{-A(t),f^{L,+}(e^L), f^{R,-}(e^R)\} \right\}, 
$$
where $Q= [0,R^L]\times [0,R^R]$,  $A_0 =- \min\{\max f^L, \max f^R\}$ and $f^+$ and $f^-$ denote respectively the  increasing and decreasing parts of a given  map $f$. 
We note that is  known, see, for example, \cite{andreianov}, that, for piecewise constant controls $A:[0,T]\to [A_0, 0]$, there exists a unique  entropy solution $\rho^A$ to \eqref{eq.CL}.

\vs

The general optimal control problem  is the  minimization  over all piecewise constant maps $A:[0,T]\to \R$ of a quantity of the form 
\be\label{basicpb}
 \int_0^T \int_\R \psi(x,t,\rho^A(x,t),A(t))dt
\ee
where $\psi:\R\times [0,T]\times \R_+\times [A_0,0]\to \R$ is a continuous function with a compact support  and  convex in the last variable. 
\vs
Models  about  the optimal control of traffic flow have attracted  considerable  attention recently; see, for example, the  monograph \cite{BDmGGP22} and to the survey papers \cite{BCGH, Pi23}.  
In fact, several optimal control problems of the type discussed above have been looked at  in the literature, some 
of which having  for more general junction conditions; 
 see, for instance, \cite{BDmGGP22, BH11, BH12, BH13, BrNg15, CDPR07, CGR11, GHKL05, HK03, TCD21}. In particular, \cite{GHKL05} derives optimality conditions assuming the smoothness of the optimal solution. In \cite{CDPR07}, the controls are assumed to be  constant in time. 
Reference  \cite{CGR11} discusses how to minimize the queue length, the total variation of traffic speed (see also \cite{CG04}), and general functions of the density as in \eqref{basicpb} for problems on the half-line.  As it is explained in \cite{CGR11}, this later case includes the ``travel time'', that is, the average time needed to reach a given position. The main result of \cite{CGR11} is the existence of optimal solutions when controlling the flux at the entry of the road. In \cite{ACCG18} (see also \cite{HK03}), the authors study problems on a junction. Their aim is to  optimize integrals in time of a function of the flux at the node, namely, expressions of the form $\psi( f^L(\rho^A(0^-,t)),f^R(\rho^A(0^+,t))$
and prove the existence of an optimal solution under  a total variation constraint on the control. 
\vs

Our objective  is to show that the problem consisting in minimizing  \eqref{basicpb} has a (relaxed) minimum (Theorem \ref{thm.existence}), to give optimality conditions in the case where $\psi$ is a linear function of $\rho^A$ or of $f^{L/R}(\rho^A)$ (Theorem \ref{lem.OC}) and to show that, under suitable structure conditions on $\psi$ and the initial condition (including the case of maximizing the flux on a space-time box), the relaxed control  is piecewise constant and takes values in $\{0,A_0\}$ (Theorem \ref{thm.mainexTOT}). Note that this latter result is meaningful in the context of  traffic flows. It shows that it is more efficient to stop the traffic by using traffic lights  with only a finite number of switches than to slow it down by reducing the speed. We underline however that the optimal solution depends  on the choice of both the initial condition and the cost function.  
\vs

The existence of a solution is a new result since it does not require any BV bound on the control in contrast to other results in the literature \cite{ACCG18, CDPR07}. The price to pay is that the optimal control is then a priori in $L^\infty$  and not BV. As a result,  it is necessary  to interpret the solution in a suitable way (see below). Note, however, that, under the assumptions of Theorem \ref{thm.mainexTOT}, the relaxed optimal control turns out to be piecewise constant, and thus a ``classical'' control. With the exception of \cite{GHKL05}, where the optimal solution,  that is, the pair $(A,\rho^A)$) is supposed to be smooth, and \cite{CDPR07}, in  which controls are assumed to be constant in time, there is no optimality conditions  in the literature. Here we obtain such  conditions, at least under specific structure assumptions. Finally,  we derive from these conditions that the optimal control is bang-bang for the problem of maximizing the flux in a space-time box. 
\vs

As the direct analysis of the minimization problem in \eqref{basicpb} is delicate, we rely on the representation of the solution $\rho^A$ in terms of Hamilton-Jacobi equations. This approach is classical in the literature; see, for example,  \cite{IMZ} and the references therein or \cite[section 3.3.1]{BDmGGP22}  for a comprehensive account. It amounts to replacing the density $\rho^A$  by the Moskowitz function $u^A=u^A(x,t)$, which, roughly speaking, counts the number of vehicles passed by location $x$ before time $t$. Then $\rho^A=-u^A_x$ and  $u^A$ is the unique viscosity solution of the Hamilton-Jacobi equation on the junction problem
\be\label{main.HJ}
\left\{\begin{array}{ll}
u^A_t(x,t) +H^L(u^A_x(x,t))= 0  \  \; \  \text{in}\; \ \ (-\infty,0)\times (0,T),\\
u^A_t(x,t) +H^R(u^A_x(x,t))= 0 \; \ \ \text{in}\;  \ \ (0, \infty)\times (0,T),\\
u^A_t(0,t) + \max\{ A(t), H^{L,+}(u^A_x(0^-,t)), H^{R,-}(u^A_x(0^+,t))\} =0 \ \ \; \text{in} \; \ \ (0,T), \\
u^A(x,0)= u_0(x) \ \  \;\text{in} \; \ \ R,
\end{array}\right.
\ee
where $H^{L/R}(v)=-f^{L/R}(-v)$. Note that $H^{L/R}:(-R^{L/R}, 0)\to \R$ are smooth uniformly convex Hamiltonians with $H^{L/R}(-R^{L/R})=H^{L/R}(0)=0$.  The initial condition $u_0:\R\to\R$ is a Lipschitz map such that $u_{0,x} =-\rho_0$, $(u_0)_x\in (-R^L,0)$ on $(-\infty,0)$ and  $u_{0,x}\in (-R^R,0)$ on $(0,\infty)$. The (now relaxed) control $A$ is a measurable map taking values in $[A_0,0]$, where $A_0=\max\{\min H^L, \min H^R\}$.
\vs
The optimization problem \eqref{basicpb} then consists in minimizing  over measurable controls $A$, taking values in $[A_0,0]$, the cost
\be\label{optipb}
\mathcal J(A)= \int_0^T \int_\R \tilde \psi(x,t,u^A_x(x,t),A(t))dt,
\ee
where $\tilde \psi(x,t,\rho,A)= \psi(x,t,-\rho,A)$ is continuous function with a compact support which is convex in the last variable. We summarize saying that we study the optimization problem
\be\label{takis10}
\min \{\mathcal J(A): A:(0,T)\to [A_0,0] \ \text{measurable} \},
\ee
with  $\mathcal J(A)$ as in \eqref{optipb}.
\vs
Note that we restrict our analysis to a single junction on the line. The reason for this is that it is known from \cite{CFM23} that the junction conditions on a more general junction can be in general much more intricate. For example,  the control parameter in the case of one entry and two exit lines could be at each time a function describing how to dispatch the entry to the exits. In addition, the method of proof developed in this paper break down in this three  lines model because the equivalence between conservation law and Hamilton-Jacobi equations is lost. \vs
The second strong restriction of the paper is the structure of $\psi$ to obtain optimality conditions. If the existence of optimal solution can be derived under the general conditions given above, in order to study the optimality conditions we need to assume the cost functional to be written in the form 
\be\label{hyp.formpsi}
\mathcal J(A)= \int_0^T \int_\R  [\phi(x,t)u^A(x,t)+ f(A(t))]dt,
\ee
for a suitable continuous map $\phi:\R\times [0,T]\to \R$ with a compact support and $f:[A_0,0]\to \R$ of class $C^1$ and convex.  The reason for this restriction is that the derivative $u^A_x$ has a very singular dependence with respect to the variations of $A$. 
We show, however,  that the map $A\to u^A(x,t)$ has directional derivatives in many directions, which allows to derive optimality conditions. The cost function \eqref{optipb} can actually be written in the form \eqref{hyp.formpsi} when the map $\psi$ is an affine function of $\rho$ or of $f(\rho)$. This includes the problems consisting in maximizing the flux or minimizing the travel time. 
%
\vs
In this Hamilton-Jacobi setting, our main results are: the existence of an optimal minimizer $A\in L^\infty([0,T], [A_0,0])$ for the optimization problem \eqref{takis10} (Theorem \ref{thm.existence}). Following \cite{imbert-monneau} (see also \cite{AOT15, BCbook, LS1, LS}), it is known that, given a piecewise constant control $A$ taking values in $[A_0,0]$, the Hamilton-Jacobi equation \eqref{main.HJ} has a unique viscosity solution  $u^A$ and that $\rho^A= -u^A_x$ is the unique entropy solution of \eqref{eq.CL} (see \cite{CFGM}). The definition of a solution $u^A$, when $A$ is just measurable, is made by  using a representation formula (see Definition \ref{def.solHJ}), the equivalence between the two formulations (viscosity solution and representation formula) being known when $A$ is piecewise constant.  Assuming that the cost function $\mathcal J$ takes the form \eqref{hyp.formpsi}, we also obtain an optimality condition on the minimizer  (Theorem \ref{lem.OC}). We then  study thoroughly the nature of the maximizer when the problem consists in maximizing the flux on a space-time box of the form $[x_1,x_2]\times [t_1,t_4]$. Under suitable assumptions on the data, we show that the optimal control is piecewise constant and takes its values in $\{0,A_0\}$ (Theorem \ref{thm.mainexTOT}). We also give an example in which constant controls are not optimal (Proposition \ref{prop.exAnottrivial}). 
\vs
The paper is organized as follows. We first introduce our standing assumptions and notations (Section \ref{sec.hyp}). Then we prove the existence of a minimizer to \eqref{optipb} (Section \ref{sec.exist}).  Section \ref{sec.OC} gives the  optimality condition, while  Section \ref{sec.bangbang} is devoted to  problem of maximization of the flux in a box.

\section{Main definitions and assumptions} \label{sec.hyp}

\noindent {\bf Standing assumptions.} Throughout the paper we assume that 
\begin{equation}\label{takis1}
\left\{\begin{array}{ll}
H^L:\R\to \R \ \text{and} \  H^R:\R\to \R \ \text{ are uniformly convex  and}\\[1mm]
\text{ of class $C^2-$Hamiltonians}  \text{ and there exist} \ R^L \ \text{and} \  R^R \  \text{such that }\\[1mm]
H^L(-R^L)= H^L(0)=0=H^R(-R^R)=H^R(0),
\end{array}\right.
\end{equation}
and
\begin{equation}\label{takis2}
\left\{\begin{array}{ll}
 \text{the initial condition $u_0:\R\to \R$ is Lipschitz continuous and} \\[1mm] 
\text{$u_{0,x}\in (-R^L,0)$ in $\R_-$ and $ u_{0,x}\in (-R^R,0)$ in $\R_+$,} 
\end{array}\right.
\end{equation}
where 
\[A_0= \max\{\min_{p}H^L(p), \min_{p}H^R(p)\}. \]

We also denote by $\hat p^L\in (-R^L, 0)$ and  $\hat p^R\in (-\hat p^R,0)$ the unique point such that $H^L(\hat p^L)=\min_p H^L(p)$ and  $H^R(\hat p^R)=\min_p H^R(p)$. Note that $H^L$ is decreasing on $[-R^L,\hat p^L]$ and increasing on $[\hat p^L,0]$, and $H^R$ is decreasing on $[-R^L,\hat p^R]$ and increasing on $[\hat p^R,0]$
\vs

%
We write  $H^{L,+}$ and $H^{L,-}$ for the increasing part and decreasing part of $H^L$ respectively, that is,
$$
H^{L,+}(p)= \left\{\begin{array}{ll}
\underset{p}\min H^L(p) & \text{if}\; p\leq \hat p^L,\\
H^L(p) & \text{otherwise}, 
\end{array}\right.\qquad 
H^{L,-}(p)= \left\{\begin{array}{ll}
H^L(p) & \text{if}\; p\leq \hat p^L,\\[1.2mm]
\underset{p}\min H^L(p) & \text{otherwise}, 
\end{array}\right.
$$
and we  use a symmetrical notation for $H^R$.
\vs

%

\noindent {\bf Definition of a solution to \eqref{main.HJ}.}\label{def.solHJ} Following \cite{imbert-monneau} (see also \cite{BCbook, LS1, LS}), it is known that, for any constant flux limiter $A\in [A_0,0]$, there exists a unique viscosity solution $u^A$ to \eqref{main.HJ}, which can be represented as 
$$
u^A(x,t)= \inf_{\gamma\in H^1, \; \gamma(t)=x} \big[\int_0^t \big(L^R(\dot \gamma){\bf 1}_{\{\gamma>0\}} +L^L(\dot \gamma){\bf 1}_{\{\gamma<0\}} -A{\bf 1}_{\{\gamma=0\}} \big)ds +u_0(\gamma(0))\big],
$$
where the Lagrangians $L^L$ and $L^R$ are defined, for any $\alpha\in \R$,   by 
$$
L^{L}(\alpha)= \sup_{p\in [-R^L, 0]} \{\alpha p -H^{L}(p)\} \ \ \text{and}\ \  L^{R}(\alpha)= \sup_{p\in [-R^R, 0]} \{\alpha p -H^{R}(p)\}. 
$$

The existence, uniqueness and representation of a solution can easily be extended to piecewise constant maps $A:[0,T]\to [A_0,0]$; see, for instance, the construction in \cite{CFM23}. Here we need to work with measurable maps $A: [0,T]\to [A_0,0]$. Although it is possible to extend the notion of viscosity solution to this more general framework (see the  paper in preparation \cite{Brprepa}), we will not use this point of view here and only work with the representation formula. 

\begin{defn} Let $A:[0,T]\to [A_0,0]$ be measurable. Then we denote by $u^A$ the map 
\be\label{calcvaruA}
u^A(x,t)= \inf_{\gamma\in H^1(0,t), \; \gamma(t)=x} J^A(\gamma),
\ee
where 
$$
J^A(\gamma) = \int_0^t (L^R\big(\dot \gamma(s)){\bf 1}_{\{\gamma(s)>0\}} +L^L(\dot \gamma(s)){\bf 1}_{\{\gamma(s)<0\}} -A(s){\bf 1}_{\{\gamma(s)=0\}} \big)ds  +u_0(\gamma(0)).
$$
\end{defn}

It is easy to see that  dynamic programming gives that $u^A$ is Lipschitz continuous and satisfies \eqref{main.HJ} in the viscosity sense in $(\R\backslash \{0\})\times (0,T)$.  Note also that $u_x^A\in [-R^L,0]$ and $u^A_t \in [\min H^L,0]$ a.e. in $\{x<0\}$, while $u_x^A\in [-R^R,0]$ and $u^A_t \in [\min H^R,0]$ a.e. in $\{x>0\}$.

\section{Existence of an optimal control} \label{sec.exist}
We prove here the existence of a minimizer for the  optimal control problem \eqref{takis10}. This also leads to  the existence of optimal trajectories for \eqref{calcvaruA}. 

\subsection{The existence result}

Throughout the section we assume that
\begin{equation}\label{takis100}
\begin{split}
& \text{the cost function} \   \tilde \psi:\R\times [0,T]\times \R\times [A_0,0]\to \R  \  \text{is measurable and bounded,}\\ 
& \text{has  compact support, and is continuous  in the last two variables}\\
& \text{and is convex in the last variable.}
\end{split}
\end{equation}
The main result of this section is the following theorem.

\begin{thm}\label{thm.existence}  Assume \eqref{takis1}, \eqref{takis2} and \eqref{takis100}. 
Then, the  problem \eqref{takis10} admits at least one solution. 
\end{thm}

The proof  is given at the end of the section. The key remark is the continuity of the map $A\to u^A$ with respect to the weak-$\ast$ convergence of $A$, a fact which needs justification in view of the highly nonlinear way in which  $u^A$ depend on $A$. 
In view of the uniform (in $A$) local semiconcavity of $u^A$ in $(\R\backslash\{0\})\times (0,T)$, this continuity then easily implies the continuity of the map $A\to u^A_x$. 

\vs

The proof of the continuity of $u^A$ in $A$ is based on the following lemma, also proved at the end of the section, which implies  that   we can always assume that any $\ep-$suboptimal path $\gamma$ for $u^A(x,t)$ is piecewise linear and may vanish on an interval.   

\begin{lem}\label{takis11BIS}
Assume \eqref{takis1} and  \eqref{takis2} and fix  a measurable $A: [0,t] \to [A_0, 0]$. For any  $\gamma\in H^1((0,t);\R)$, there exists $\tilde \gamma\in H^1((0,t);\R)$  and $0\leq a\leq b\leq t$ such that \\
(i)~$\tilde \gamma(t)=\gamma(t)$,\\
(ii)~$ J^A(\tilde \gamma)\leq  J^A(\gamma)$, and\\
(iii)~either $\tilde \gamma$ is a nonvanishing straight-line, or $\tilde \gamma$ is affine and does not vanish on $(0,a)$ and on $(b,t)$, and $\tilde \gamma \equiv 0$ on $[a,b]$. \\
Moreover, $ J^A(\tilde \gamma)<  J^A(\gamma)$ unless $\gamma= \tilde \gamma$. 
\end{lem}

%
\vs

We continue with the continuity property of $u^A$ in $A$.

\begin{prop}\label{prop:weakstab} Assume that a sequence $(A^n)_{n\in \N}$ in $L^\infty([0,T], [A_0,0])$  converges weakly-$\ast$ to $A\in L^\infty([0,T], [A_0,0])$. Then the $u^{A^n}$'s converge locally uniformly to $u^A$. In addition, for any $(x,t)\in \R\times (0,T]$, there exists an optimal trajectory $\gamma\in H^1([0,t], \R)$ with $\gamma(t)=x$ for $u^A(x,t)$. 
\end{prop}

\begin{proof} Note first that, since the  $u^{A^n}$'s are uniformly Lipschitz continuous, they converge along subsequences. Without loss of generality, we can assume that the whole sequence $(u^{A^n})$ converges to some $u$ and we check that $u= u^A$. 
\vs
We first observe that $u\leq u^A$. Indeed, for any $(x,t)\in \R\times (0,T]$ and any $\gamma\in H^1([0,t])$ such that $\gamma(t)=x$, we have, in view of the weak-$\ast$ convergence of  the $u^{A^n}$'s, that 
\begin{align*}
u(x,t) & = \underset{n\to \infty} \lim u^{A^n}(x,t)\leq \underset{n\to \infty} \lim  \int_0^t \big[L^R(\dot \gamma){\bf 1}_{\{\gamma>0\}} +L^L(\dot \gamma){\bf 1}_{\{\gamma<0\}} -A^n{\bf 1}_{\{\gamma=0\}} \big ]ds +u_0(\gamma(0))\\
&=  \int_0^t \big[L^R(\dot \gamma){\bf 1}_{\{\gamma>0\}} +L^L(\dot \gamma){\bf 1}_{\{\gamma<0\}} -A{\bf 1}_{\{\gamma=0\}} \big]ds +u_0(\gamma(0))= J^A(\gamma).
\end{align*}
Taking the infimum over $\gamma$ gives the claim. 
\vs 
For the other inequality, we fix $(x,t)\in \R\times (0,T]$. 
Let $\gamma^n$ be a $1/n-$minimizer for $u^{A^n}(x,t)$,
that is 
$$
J^{A^n}(\gamma^n) \leq u^{A^n}(x,t) +1/n. 
$$

Following Lemma~\ref{takis11BIS}, we may assume that either $\gamma^n$ is affine and does not vanish on $(0,t)$, or  there exist $0\leq a^n \leq   b^n<t$ such that $\gamma^n$ is affine and does not vanish in $(b^n,1]$, $\gamma$ is affine and does not vanish in $[0,a^n)$ and $\gamma^n\equiv 0$ in $[a^n,b^n]$. 
\vs 
Next we show that $\gamma^n(0)$ can be assumed to be bounded. For simplicity, we assume that $\gamma^n$ vanishes on $(0,t)$ with $a^n>0$, as the other cases can be treated with similar and simpler argument. Recall that 
\be\label{lkzjensdf}
L^L(\alpha)= \sup_{p\in [-R^L, 0]} \{ p\alpha -H^L(p) \}= \left\{\begin{array}{ll}
0 & \text{if }\; \alpha \geq (H^L)'(0),\\
- R^L \alpha & \text{if }\; \alpha \leq (H^L)'(-R^L),
\end{array}\right.
\ee
and the symmetric equality for $L^R$. 
\vs
Since  $\gamma^n(a^n)=0$, we have $\gamma^n(s)=(a^n-s)\gamma^n(0)/a^n$ on $[0,a^n]$. Thus
$$
J^{A^n}(\gamma^n) = a^n(L^R(-\gamma^n(0)/a^n){\bf 1}_{\{\gamma^n(0)>0\}}+ L^L(-\gamma^n(0)/a^n){\bf 1}_{\{\gamma^n(0)<0\}})+u_0(\gamma^n(0))+R^n, 
$$
where 
$$
R^n = \int_{b^n}^t (L^R(\dot \gamma^n){\bf 1}_{\{\gamma>0\}}+L^L(\dot \gamma^n){\bf 1}_{\{\gamma<0\}}) -\int_{a^n}^{b^n} A^n(s)
$$

is bounded since $\gamma^n$ is affine between $b^n$ and $t$,  vanishes at $b^n$ and is equal to $x$ at $t$. 
\vs
On the other hand, if $-\gamma^n(0)/a^n\leq (H^R)'(-R^R)$, that is, $\gamma^n(0)\geq y^n=-a^n (H^R)'(-R^R)>0$, then, as $u_0$ satisfies $u_{0,x}\geq -R^R$ in $(0, \infty)$ and \eqref{lkzjensdf} holds, we find 
\begin{align*}
 J^{A^n}(\gamma^n) & = a^n L^R(-\gamma^n(0)/a^n) +u_0(\gamma^n(0))+R^n  \\
 & \geq  R^R \gamma^n(0)+u_0(y^n)+(-R^R)( \gamma^n(0)-y^n) +R^n \\
 & = a^n L^R(-y^n/a^n) +u_0(y_n) + R^n, 
 \end{align*}
 which shows that we can replace $\gamma^n(0)$ by $y^n= -a^n (H^R)'(-R^R)$ in this case. 
 \vs
 In a symmetric way, if $-\gamma^n(0)/a^n\geq (H^L)'(0)$, that is, $\gamma^n(0)\leq z^n=-a^n (H^L)'(0)<0$, then, as $u_0$ is nonincreasing and \eqref{lkzjensdf} holds, 
\begin{align*}
 J^{A^n}(\gamma^n) & = a^n L^L(-\gamma^n(0)/a^n) +u_0(\gamma^n(0))+R^n \geq u_0(z^n)+R^n \\
 & = a^n L^R(-z^n/a^n) +u_0(z_n) + R^n, 
 \end{align*}
which proves that we can replace  $\gamma^n(0)$ by $z^n= -a^n (H^L)'(0)$. 
\vs
In conclusion, we can assume that $\gamma^n(0)$ is bounded. 
\vs

Taking a subsequence,  still labelled for simplicity in the same way, we now assume that the $\gamma^n$'s  converge on $[0,t]$ to some $\gamma$ and that $a^n$'s and $b^n$'s converge to some $a$ and $b$ respectively, with $0\leq a\leq b<t$. 
Note  that the $\dot{\gamma}^n$'s  converge to $\dot \gamma$ in $L^1$. 
%
%
\vs

If $a=0$ or $a>0$ and $\gamma$ does not vanish in $[0,a)$,  and if $b=t$ or if $\gamma$ does not vanish in $(b,t]$, it follows that   $\{\gamma=0\}=[a,b]$ is the limit of the interval $\{\tilde \gamma^n=0\}=[a^n, b^n]$ and the same holds for the sets $\{\gamma<0\}$ and $\{\gamma>0\}$. Since  the $\dot{ \gamma}^n$'s converge to $\dot \gamma$ in $L^1$, 
we can pass directly to the limit in the definition of $u^{A^n}(x,t)$ to get
\begin{align*}
u(x,t) & = \int_0^t L^L(\dot \gamma(s)){\bf 1}_{\{\gamma(s)<0\}}+  L^R(\dot \gamma(s)){\bf 1}_{\{\gamma(s)>0\}} - A (s) {\bf 1}_{\{\gamma(s)=0\}} \ ds + u_0(\gamma(0))\\ & \geq u^A(x,t), 
\end{align*}
which proves the equality $u(x,t)=u^A(x,t)$ and the fact that $\gamma$ is optimal for $u^A(x,t)$ in this case. 
\vs

We now suppose that $a>0$ and $\gamma$ vanishes  somewhere in $(0,a)$, but that $b=t$ or that $\gamma$ does not vanish in $(b,t]$. 
Then, since  $\gamma(a)=0$ and $\gamma$ is affine, $\gamma \equiv0$ on $[0,a]$. Writing $\gamma^n(s)=\alpha^ns+\beta^n$ on $(0,a^n)$, with $\beta^n\neq 0$ since $\gamma^n(0)\neq 0$ and $\alpha^n\neq 0$ since $\gamma^n(a^n)=0$, we have  that $\alpha^n\to 0$ and $\beta^n\to 0$. 
\vs
Assume next, to fix the ideas, that $\beta^n<0$ for any $n$ up to a subsequence. In view of the definition of $u^{A^n}(x,t)$, we have 
\[
u^{A^n}(x,t)= a^n  L^L(\alpha^n) - \int_{a^n}^{b^n} A^n (s) ds + u_0(\beta^n) + R^n,
\]
with 
$$
R^n =  \int_{b^n}^t \big[L^L(\dot \gamma^n(s)){\bf 1}_{\{\gamma^n(s)<0\}}+  L^R(\dot \gamma^n(s)){\bf 1}_{\{\gamma^n(s)>0\}}\big] \ ds. 
$$

The argument developed above yields that the $R^n$'s converge to 
$$
R=  \int_b^t \big[L^L(\dot \gamma(s)){\bf 1}_{\{\gamma(s)<0\}}+  L^R(\dot \gamma(s)){\bf 1}_{\{\gamma(s)>0\}}\big]\ ds . 
$$
Thus 
$$
u(x,t)= aL^L(0) - \int_a^b A(s)ds + u_0(0)+ R.
$$
As $L^L(0)= -\min_p H^L(p)\geq  -A_0\geq -A(s)$ for a.e. $s\in [0,a]$, we get 
$$
u(x,t)\geq -\int_0^b A(s)ds + u_0(0)+ R = J^A(\gamma) \geq u^A(x,t). 
$$
This shows that  $u(x,t)=u^A(x,t)$ and the fact that $\gamma$ is optimal for $u^A(x,t)$ in this case. 
%
\vs
If  $\beta^n>0$ up to a subsequence, or if $b<t$ and $\gamma$ vanishes in $(b,t]$, we argue similarly. 
\vs
In conclusion, we have proved that $u=u^A$ and, therefore,  that the sequence $(u^{A^n})$ converges to $u^A$. In addition  we have also shown  the existence of an optimal trajectory $\gamma$ for $u^A(x,t)$ for any $(x,t)\in \R\times (0,T]$. 
\end{proof}

We continue with  the proofs of Theorem~\ref{thm.existence} and Lemma~\ref{takis11BIS}.
\vs
\begin{proof}[Proof of Theorem \ref{thm.existence}] Let $A^n$ be a minimizing subsequence. Extracting a subsequence if necessary, we can assume that the $A^n$'s converge  weakly$-\ast$  to some $A\in L^\infty([0,T], [A_0,0])$ and, therefore,  that the $u^{A^n}$'s  converge locally uniformly to $u^A$. 
\vs
Since  the Hamiltonians are uniformly convex, the  $u^{A^n}$'s are uniformly locally semiconcave in $(x,t)$ on compact subsets of $(\R\backslash\{0\})\times (0,T)$. Thus the bounded sequence $(u^{A^n}_x)_{n\in \N}$ converges in $L^1_{Loc}$ to $ u^A_x$.  The continuity of the compactly supported $\tilde \psi$ and its convexity in the last variable,  along standard lower semicontinuity arguments, yield  that $A$ is a minimum of $\mathcal J$. 
\end{proof}

\begin{proof}[Proof of Lemma~\ref{takis11BIS}] Fix $\gamma\in H^1(0,t)$. If $\gamma$ does not vanish on $[0,t]$, then we denote $\tilde \gamma$ be the linear interpolation between $\gamma(0)$ and $\gamma(t)$. By the strict convexity of $L^L$, if $\gamma$ remains negative, and of $L^R$, if $\gamma$ remains positive, we have immediately that $ J^A(\tilde \gamma)\leq  J^A(\gamma)$, with an equality only if $\gamma=\tilde \gamma$.

Let us now assume that $\gamma$ vanishes in $[0,t]$ and let 
\[a=\sup\{a' \in (0,t)\; : \text{$\gamma$ does not vanish in $[0,a')$}\},\]
\[b=\inf\{b'\in (0,t) \; : \text{$\gamma$ does not vanish in $(b',t]$}\}.\]
By convention, we set $a=0$ if $\gamma(0)=0$ and/or $b=t$ if $\gamma(t)=0$. Note that,  since $\gamma$ vanishes in $[0,t]$, $0\leq a\leq b\leq t$. 
\vs
We define $\tilde \gamma$ as follows: if $a>0$, then $\tilde \gamma$ is  on $[0,a]$ the linear interpolation between $\gamma(0)$ and $0$. On $[a,b]$, we set $\tilde \gamma\equiv 0$. Finally, if $b<t$, then $\tilde \gamma$ is  on $[b,t]$ the linear interpolation between $0$ and $\gamma(t)$.  Note that $\tilde \gamma$ is affine and does not vanish on $(0,a)$ and on $(b,t)$, and $\tilde \gamma\equiv 0$ on $[a,b]$. In addition  $\tilde \gamma(t)=\gamma(t)$.\vs 
Using again the strict convexity of $L^L$ and $L^R$ we have immediately that 
$$
\int_{[0,a]\cup[b,t]} \big[L^L(\dot \gamma){\bf 1}_{\{\gamma<0\}}+ L^R(\dot \gamma){\bf 1}_{\{\gamma>0\}}\big] ds \geq  
\int_{[0,a]\cup[b,t]} \big[L^L(\dot{\tilde \gamma}){\bf 1}_{\{\tilde \gamma<0\}}+ L^R(\dot{\tilde  \gamma}){\bf 1}_{\{\tilde \gamma>0\}}\big] ds,
$$
with an equality only if $\gamma=\tilde \gamma$ on $[0,a]\cup[b,t]$. \vs 
We now compare the cost for $\gamma$ and $\tilde \gamma$ on $[a,b]$. The open set $\{\gamma\neq 0\}\cap (a,b)$ is covered by a countable union of disjoint open intervals, and $\gamma=0$ at the end points of each of these intervals.
\vs
Let $(c,d)$ be one of these intervals and assume, without any loss of generality, that $\gamma >0$ there. Then using the strict convexity of $L^R$ and that $\gamma(c)=\gamma(d)=0$, we find 
$$
\int_c^d L^R(\dot \gamma(s))ds > \int_c^d L^R(0) ds +\int_a^b (L^R)'(0)(\dot \gamma^n(s)-0)ds= \int_c^d L^R(0) ds,
$$
which, in view of the fact that  $L^R(0)= \sup_{p\in [-R^0,0]} -H^R(p)\geq -A_0\geq -A(s)$ a.e. in $s$, further  yields
$$
\int_c^d L^R(\dot \gamma(s))ds > - \int_c^d A(s)ds.
$$
Arguing similarly for any of the other open intervals, we see that 
\begin{equation*}
\begin{split}
\int_a^b \big[L^R(\dot \gamma(s)){\bf 1}_{\{\gamma>0\}}+ L^L(\dot \gamma(s)){\bf 1}_{\{\gamma<0\}}-A(s) {\bf 1}_{\{\gamma=0\}}\big]ds\\[1mm]
 \geq  - \int_a^b A(s)ds= - \int_a^b A(s){\bf 1}_{\{\tilde \gamma=0\}}ds, 
\end{split}
\]
with an equality only if $\gamma\equiv 0$ in $[a,b]$. 
\vs
This shows that $J^A(\tilde \gamma)\leq J^A(\gamma)$ with  equality only if $\gamma=\tilde \gamma$. 
\end{proof}
 
\subsection{Structure of optimal trajectories} 

Given an admissible control, we now  investigate the existence optimal trajectories for the optimal control problem \eqref{calcvaruA} and explain their structure. 
\vs

The representation formula \eqref{calcvaruA} implies that,   for any admissible control $A$, we have\\
 $u^A\leq u^{A_0}$. We can therefore expect that, for points $(x,t)$ at which $u^A(x,t) < u^{A_0}(x,t)$, optimal trajectories have specific structure. This is exactly what the second part of the following proposition says.

\begin{prop}\label{prop.structopti} Assume \eqref{takis1} and \eqref{takis2}. For any $A\in L^\infty([0,T], [A_0,0])$, there exist an optimal path $\hat \gamma_{x,t}$ in \eqref{calcvaruA} which is 
either affine and does not vanish on $(0,t)$, or  there exist $0\leq t_1 \leq t_2\leq T$ such that $\hat \gamma_{x,t}$ is affine and does not vanish on $(0,t_1)$ and on $(t_2,T)$, and vanishes on $[t_1,t_2]$. 
Moreover, 
if $(x,t)\in \R\times (0,T]$ is such that $u^A(x,t)<u^{A_0}(x,t)$ and $\gamma$ is optimal for $u^A(x,t)$, then $\gamma$ vanishes  in $(0,t)$ and, if 
$$
\tau^+=\sup\{s\geq 0, \; \gamma(s)=0\} \ \ \text{and} \ \  \tau^-=\inf\{s\geq 0, \; \gamma(s)=0\},
$$
then  $\tau^-<\tau^+$, $u^{A}(0,\tau^+)<u^{A_0}(0,\tau^+)$ and  $u^{A}(0,\tau^-)=u^{A_0}(0,\tau^-)$. 
%
\end{prop} 


\begin{proof} Combining Lemma \ref{takis11BIS} with Proposition \ref{prop:weakstab} easily gives the first claim.  
\vs
We now assume that $u^A(x,t)<u^{A_0}(x,t)$. If $\gamma$ does not vanish or if $\tau^-=\tau^+$, then 
$$
u^A(x,t) = \int_0^t \big[L^R(\dot \gamma){\bf 1}_{\{\gamma>0\}}+L^L(\dot \gamma){\bf 1}_{\{\gamma<0\}}\big]ds +u_0(\gamma(0))= J^{A_0}(\gamma) \geq u^{A_0}(x,t),
$$
which contradicts the assumption $u^A(x,t)<u^{A_0}(x,t)$. Thus $\tau^-$ and $\tau^+$ exist and $\tau^-<\tau^+$. 

\vs
Assume that $u^{A}(0,\tau^+)=u^{A_0}(0,\tau^+)$. Dynamic programming yields that 
\begin{align*}
u^{A_0}(x,t)  & \leq \int_{\tau^+}^t \big[L^R(\dot \gamma){\bf 1}_{\{\gamma>0\}}+L^L(\dot \gamma){\bf 1}_{\{\gamma<0\}}\big]ds + u^{A_0}(0,\tau^+)\\
& = \int_{\tau^+}^t \big[L^R(\dot \gamma){\bf 1}_{\{\gamma>0\}}+L^L(\dot \gamma){\bf 1}_{\{\gamma<0\}}\big]ds + u^{A}(0,\tau^+)= u^A(x,t), 
\end{align*}
which contradicts the assumption $u^A(x,t)<u^{A_0}(x,t)$. So $u^{A}(0,\tau^+)<u^{A_0}(0,\tau^+)$.
\vs
We now check the equality $u^{A}(0,\tau^-)=u^{A_0}(0,\tau^-)$. If $\tau^-=0$, the result is obvious since $u^A(\cdot,0)=u^{A_0}(\cdot, 0)=u_0$. We now assume that $\tau^->0$. Then $\gamma$ does not vanish on $[0, \tau^-)$ and thus, by dynamic programming,  
$$
u^{A_0}(0,\tau^-)\leq \int_0^{\tau^-}(L^R(\dot \gamma){\bf 1}_{\{\gamma>0\}}+L^L(\dot \gamma){\bf 1}_{\{\gamma<0\}})ds + u_0(\gamma(0))= u^{A}(0, \tau^-) \leq u^{A_0}(0,\tau^-), 
$$
which proves the equality.
\end{proof}

\begin{rmk} \label{rem.A=A0inuA=uA0} Note that there are in general infinitely many optimal controls for $\mathcal J$ in \eqref{takis10}. Thus requiring  additional conditions on the optimal control is natural. The additional condition in the next result consisting in  minimizing $\int_0^T  A$ is especially suitable because its means that one wants to perturb the traffic as little as possible. Note that $A\equiv A_0$ means no speed limit or traffic light. 
\end{rmk}

A consequence of the structure of the optimal trajectories is the following remark.
\begin{lem}\label{lem.barA=A0} If  \eqref{takis1}, \eqref{takis2} and \eqref{takis100} hold and $f$ is nondecreasing, then any minimizer $\bar A$ which in addition minimizes $\int_0^T \bar A$ among all the minimizers satisfies 
$$
\bar A(t) = A_0 \qquad \text{a.e. in } \{u^{\bar A}(0,\cdot)=u^{A_0}(0,\cdot)\}.
$$
\end{lem}

%
\begin{proof} Define
$$
\tilde A(t) = \left\{ \begin{array}{ll}
A_0 & \text{if } u^{\bar A}(0,t)=u^{A_0}(0,t),\\[1mm]
\bar A(t) & \text{otherwise}
\end{array}\right.
$$
It follows, as we show below,   that 
\begin{equation}\label{takis110}
u^{\bar A}= u^{\tilde A} \ \text{ in} \  \R\times (0,T).
\end{equation}
\vs
Assuming \eqref{takis110},  we first  complete the proof of the lemma. Indeed, $f$ being nondecreasing yields $\mathcal J(\bar A)\geq  \mathcal J(\tilde A)$. Thus $\tilde A$ is also a minimizer, $\bar A\geq \tilde A$ by construction and, by assumption on $\bar A$,  $\int_0^T \bar A\leq \int_0^T \tilde A$. It follows that  $\bar A=\tilde A$  equals $A_0$ on $ \{u^{\bar A}(0,\cdot)=u^{A_0}(0,\cdot)\}$. 
\vs

We now prove \eqref{takis110}. Since $u^{\bar A}\leq u^{\tilde A}\leq u^{A_0}$, and thus $u^{\bar A}= u^{\tilde A}$ in $\{u^{\bar A}=u^{A_0}\}$,  we  fix $(x,t)\in (\R\backslash\{0\})\times (0,T)$  such that $u^{\bar A}(x,t)<u^{A_0}(x,t)$.  
\vs
Let $\gamma$ be optimal for $u^{\bar A}(x,t)$. By Proposition \ref{prop.structopti}, we know that there exists $0\leq t_1 \leq  t_2\leq t$ such that $\gamma$ is  affine and nonvanishing on $(0,t_1)$ and $(t_2,t)$ and vanishes on $[t_1,t_2]$, with $u^{\bar A}(0,t_1)=u^{A_0}(0,t_1)$ and $u^{\bar A}(0,t_2)<u^{A_0}(0,t_2)$.
\vs 
 Let $\tau= \sup\{s\; : u^{\bar A}(\gamma(s),s)=u^{A_0}(\gamma(s),s)\}$. Then $t_1\leq \tau< t_2$, 
 $\gamma(\tau)=0$  and $u^{\bar A}(\gamma(\tau ),\tau)=u^{A_0}(\gamma(\tau),\tau)$, so that $u^{\bar A}(\gamma(\tau ),\tau)=u^{\tilde A}(\gamma(\tau),\tau)$. Moreover $u^{\bar A}(0,s)< u^{ A_0}(0,s)$ on $(\tau,t_2)$ and thus $\tilde A=\bar A$ on $(\tau,t_2)$. By dynamic programming, 
\begin{align*}
u^{\bar A}(x,t) & = \int_{t_2}^t L^R(\dot \gamma){\bf 1}_{\gamma>0}+ \int_{t_2}^t L^L(\dot \gamma){\bf 1}_{\gamma<0} -\int_{\tau}^{t_2} \bar A(s)ds + u^{\bar A}(\gamma(\tau), \tau)\\
&   =\int_{t_2}^t L^R(\dot \gamma){\bf 1}_{\gamma>0}+ \int_{t_2}^t L^L(\dot \gamma){\bf 1}_{\gamma<0}-\int_{\tau}^{t_2} \tilde A(s)ds + u^{\tilde A}(\gamma(\tau), \tau)
\\ 
& 
\geq u^{\tilde A}(x,t)\geq u^{\bar A}(x,t).
\end{align*}
Hence  $u^{\bar A}(x,t)= u^{\tilde A}(x,t)$, which  gives the result. 

\end{proof}

\section{Optimality conditions} \label{sec.OC}

In order to derive the optimality conditions for the problem consisting in minimizing \eqref{optipb}, we assume that  $\tilde \psi$ is linear in $u_x$ or in $H(u_x)$ with separated dependence in $A$. Then \eqref{optipb} can be rewritten (after integration by parts) in the form
\be\label{costf2}
\mathcal J(A)= \int_0^T \int_\R \phi(x,t) u^A(x,t)dxdt+ \int_0^T f(A(t))dt,
\ee
where 
\be\label{hypphif}
\begin{array}{c}
\text{$\phi:\R\times [0,T]\to\R$ is a continuous map with a compact support}\\
\hskip-1.7in\text{ and $f:[A_0,0]\to\R$ is convex and $C^1$.}
\end{array}
\ee 
The function  $f$ could also depend on time. We omit this dependence for simplicity of notation. Throughout this section, we fix $\bar A$ a minimizer for $\mathcal J$. 
\vs
Note that, in the case $f\equiv 0$,  the problem is interesting only if $\phi$ changes sign. Indeed, if $\phi$ is nonpositive, then the optimal solution is $\bar A\equiv A_0$, while if $\phi$ is nonnegative, then $\bar A\equiv 0$ since, for any admissible $A$, $u^{0}\leq u^A\leq u^{A_0}$. 

\subsection{Examples} We  list a few example leading to functionals of the form \eqref{costf2}. 

\subsubsection{Maximization of a weighted density} 

The goal is  to optimize the density against a smooth compactly supported weight $\xi$, that is,
\[
\mathcal J(A)= \int_0^T\int_\R \xi(x,t) \rho^A(x,t)dxdt. 
\]
\vskip-.1in
Recalling that $\rho^A=-u^A_x$, the functional  can be rewritten, after integration by parts, as 
\vskip-.05in
$$
\mathcal J(A)=-\int_0^T\int_\R \xi_x(x,t) u^A(x,t)dxdt, 
$$
\vskip-.075in
which is of the form \eqref{costf2}. 

\subsubsection{Maximization of the flux}

In traffic problems, it is  natural to optimize the flux on some portion of the exit line. In other words, given  a space-time box $(x_1,x_2)\times (t_1,t_4)$, the problem is to  aim at maximizing 
$$
\mathcal J(A)=\int_{x_1}^{x_2}\int_{t_1}^{t_2} f(\rho^A(x,t))dxdt, 
$$
where $\rho^A$ solves \eqref{eq.CL}. In terms of Hamilton-Jacobi equations, this amounts to maximizing
$$
- \int_{x_1}^{x_2}\int_{t_1}^{t_2} H(u^A_x(x,t))dxdt= \int_{x_1}^{x_2}\int_{t_1}^{t_2} u^A_t(x,t)dxdt .  
$$
Approximating the indicatrix function ${\bf 1}_{[x_1,x_1]\times [t_1,t_4]}$ by a smooth, compactly supported function $\xi$, this boils down to minimizing
$$
\mathcal J(A)=  \int_\R\int_0^T \xi(x,t)  u^A_t(x,t)dxdt= -\int_\R\int_0^T \xi_t(x,t)  u^A(x,t)dxdt, 
$$
which is of the form \eqref{costf2}. This problem will be studied in details in Section \ref{sec.bangbang}. 

\vskip-.2in

\subsubsection{Minimization of the travel time} 

In \cite{CGR11} the authors discuss the optimization of  the averaged travel time by acting on the entry condition. Assuming that the exit road is initially empty,  that is, $\rho_0=0$ in $(0,\infty)$,  and becomes eventually empty, that is,  $\rho(\cdot, \infty)=0$, the average travel time to reach a given position $\bar x>0$ is given by
$$
\frac{1}{Q_{\rm in}} \int_{-\infty}^0 \tau_{\bar x}(x)\rho_0(x) dx = \frac{1}{Q_{\rm in}}\int_0^\infty tf^R(\rho^A(\bar x, t))dt, 
$$
where $Q_{\rm in}= \int_{-\infty}^0 \rho_0^L$ and $\tau_{\bar x}(x)$ is the time needed to reach $\bar x$ from the position $x<0$. 
\vs
If we now control the junction condition and write the problem in terms of the Hamilton-Jacobi equation, we get 
$$
\mathcal J(A) = -\frac{1}{Q_{\rm in}}\int_0^\infty tH^R(u^A_x(\bar x, t))dt= \frac{1}{Q_{\rm in}}\int_0^\infty tu^A_t(\bar x, t))dt, 
$$
where, under the constraint that the  road must be eventually  empty, we must have $u^A(\bar x, \infty)= u_{\max}=u_0(-\infty)$. 
\vs
Thus, the problem becomes to minimize
\begin{align*}
\mathcal J(A) &= \frac{1}{Q_{\rm in}}\left( \left[ t (u^A(\bar x, \cdot)-u_{\max})\right]_0^\infty - \int_0^\infty (u^A(\bar x, t)-u_{\max})dt\right) \\
&=-\frac{1}{Q_{\rm in}} \int_0^\infty (u^A(\bar x, t)-u_{\max})dt. 
\end{align*}
This quantity is of the form \eqref{costf2}, at least after a suitable regularization. However, the solution is obvious. Indeed, since, for any admissible $A$, $u^A\leq u^{A_0}$, $\mathcal J$  is clearly minimal when $A\equiv A_0$. The computation above is sketchy, but the proof can easily made rigorous.

\subsection{Statement of the optimality conditions} Before discussing the optimality conditions, we need to introduce notation on the optimal trajectories staying on $\{x=0\}$. 
\vs
\paragraph{\bf Notation for the optimal trajectories for $u^{\bar A}$} 
Let $(x,t)$ be a point of differentiability of $u^{\bar A}$ with $x\neq 0$ and $t\in (0,T)$. It follows, from the structure of optimizers in Proposition \ref{prop.structopti} and standard relations between the derivative of $u^{\bar A}$ and the derivative of the optimal solution, that  there exist $t_2\geq 0$ and $y\in \R\backslash \{0\}$ such that any optimal trajectory $\hat \gamma_{x,t}$ for  $u^{\bar A}$ is affine on $[t_2,t]$ with either $\hat \gamma_{x,t}(t_2)=0$, or $t_2=0$ and $\hat \gamma_{x,t}(0)=y$. Note that $t_2$ and $y$ are independent of the choice of the optimal solution. From Proposition \ref{prop.structopti}, we also know that, if $\hat \gamma_{x,t}(t_2)=0$, then there exists $t_1\in[0, t_2]$ such that $\hat \gamma_{x,t}$ vanishes on $(t_1,t_2)$ and is affine and does not vanish on $(0,t_1)$. Here, however, $t_1$ may depend on the optimal solution. We denote by $\hat \gamma_{x,t}^+$ the optimal solution for which $t_1$ is the smallest, and by $\hat \gamma_{x,t}^-$ the optimal solution for which $t_1$ is the largest. 
\vs

\begin{rmk}{\rm Unfortunately, it seems that the set of points $(x,t)$ for which $\hat \gamma_{x,t}^-\neq \hat \gamma_{x,t}^+$ can be of positive measure. This is in contrast with standard optimal control problems with uniformly convex Hamiltonians for which minimizers are a.e. unique. Thus it is not clear that one of the inequalities in the next  theorem  is an equality, even in the set $\{A_0<A<0\}$. We show later in  Corollary~\ref{cor.lastinter} that this can, however, be  the case in some part of $(0,T)$.
}\end{rmk}

\begin{thm}\label{lem.OC} Assume \eqref{takis1}, \eqref{takis2} and  \eqref{hypphif}, and let $\bar A$ be a minimizer for $\mathcal J$ defined by \eqref{costf2}.
Then, for a.e. $s\in [0,T]$,  
$$
\int_s^T \int_\R \phi(x,t) {\bf 1}_{\{\hat \gamma_{x,t}^+(s)=0\}}dxdt - f'(\bar A(s)) \leq 0
\ \ \text{a.e. on}\ \  \{\bar A<0\},
$$
and 
$$
\int_s^T \int_\R \phi(x,t) {\bf 1}_{\{\hat \gamma_{x,t}^-(s)=0\}}dxdt -f'(\bar A(s)) \geq 0
\ \ \text{a.e. on}\ \ \{\bar A>A_0\}. 
$$
\end{thm}


\begin{proof} 
 For $\ep\in(0,1)$, let $\beta\in L^\infty([0,T])$ be such that $\beta(s)\geq 0$ a.e. and $\beta =0$ on $\{\bar A\geq -\ep\}$. Then $\big(\bar A+h\beta\big)([0,T])\subset [A_0,0]$ for any $h\in (0,\ep/\|\beta\|_\infty)$, and, hence, $\bar A+h\beta$ is an admissible control.
\vs

Let $(x,t)$ be a point of differentiability of $u^{\bar A}$ with $x\neq 0$ and $t\in (0,T)$. By the representation formula for $u^{\bar A}(x,t)$,  which can be reduced to a finite dimensional minimization problem in view of Proposition \ref{prop.structopti}, and the envelop theorem, $A\to u^{A}$ has directional derivatives in any admissible direction, and, thus, in the direction $\beta$,  with 
 $$
 \partial_A u^{\bar A}(\beta)(x,t)= \min_{\hat \gamma_{x,t} \; \text{optimal for }\; u^{\bar A}(x,t)} 
 -\int_0^t \beta(s) {\bf 1}_{\{\hat \gamma_{x,t}(s)=0\}} ds.
 $$
 Since $\beta\geq 0$, the minimum is reached for the optimal trajectory which remains the most at $0$ and thus 
 $$
 \partial_A u^{\bar A}(\beta)(x,t)=-\int_0^t \beta(s) {\bf 1}_{\{\hat \gamma_{x,t}^+(s)=0\}} ds.
 $$
 Then, 
 \begin{align*}
0\leq  \mathcal J'(\bar A)(\beta) & = \int_0^T \int_\R \phi(x,t) \partial_A u^{\bar A}(\beta) (x,t)dxdt + \int_0^T f'(\bar A(s))\beta(s)ds \\ 
 & = - \int_0^T \int_\R \phi(x,t) \int_0^t \beta(s) {\bf 1}_{\{\hat \gamma_{x,t}^+(s)=0\}} ds  dxdt + \int_0^T f'(\bar A(s))\beta(s)ds\\
 & =  -\int_0^T\beta(s)\big[ \int_s^T  \int_{\R} \phi(x,t) {\bf 1}_{\{\hat \gamma_{x,t}^+(s)=0\}}  dxdt -f'(\bar A(s))\big]ds.
 \end{align*}
As this inequality holds  for any $\beta$ as above, it follows that 
$$
\int_s^T \int_\R \phi(x,t) {\bf 1}_{\{\hat \gamma_{x,t}^+(s)=0\}}dxdt -f'(\bar A(s))\leq 0
\ \  \text{a.e. on}\ \  \{\bar A\leq -\ep\}.
$$
We can then conclude that the first inequality holds since $\ep$ is arbitrary. The other inequality can be obtained in the symmetrical way, choosing nonpositive perturbations $\beta$. 

\end{proof} 

We introduce next the notion of the last connected component of $\{u^{\bar A}(0, \cdot)<u^{A_0}(0, \cdot)\}$. 
We say that $(a,b)$ is the last connected component of $\{u^{\bar A}(0, \cdot)<u^{A_0}(0, \cdot)\}$ if, for  any other connected component $(a',b')$ of $\{u^{\bar A}(0, \cdot)<u^{A_0}(0, \cdot)\}$, we have   $b'\leq a$. It is not always true that this last connected component  exists. We will show, however, below, that this is the case under suitable assumptions.
\vs
 It turns out that we can improve the optimality conditions on $(a,T]$, when $(a,b)$ is the last connected component of $\{u^{\bar A}(0, \cdot)<u^{A_0}(0, \cdot)\}$. This is the topic of the next corollary, which is a special case of Theorem \ref{lem.OC}.

\begin{cor}\label{cor.lastinter}  Assume \eqref{takis1}, \eqref{takis2} and  \eqref{hypphif} and  let $\bar A$ be a minimizer for $\mathcal J$ defined by \eqref{costf2}. If $(a,b)$ is the last connected component of $\{u^{\bar A}(0, \cdot)<u^{A_0}(0, \cdot)\}$, then,  for a.e. $s\in (a,T]$,  
$$
\int_s^T \int_0^\infty \phi(x,t) {\bf 1}_{\{\hat \gamma_{x,t}^+(s)=0\}}dxdt = f'(\bar A(s)) 
\ \  \text{a.e. on}\ \  \{A_0<\bar A<0\},
$$
$$
\int_s^T \int_0^\infty \phi(x,t) {\bf 1}_{\{\hat \gamma_{x,t}^+(s)=0\}}dxdt \geq f'(\bar A(s)) 
\ \  \text{a.e. on}\ \ \{\bar A=0\},
$$
and 
$$
\int_s^T \int_0^\infty \phi(x,t) {\bf 1}_{\{\hat \gamma_{x,t}^+(s)=0\}}dxdt \leq f'(\bar A(s)) 
\ \  \text{a.e. on}\ \ \{\bar A=A_0\}. 
$$
\end{cor}

\begin{proof} We argue as in the proof of Theorem \ref{lem.OC} using perturbations  $\beta$ concentrated on $(a,T)$.  Proposition \ref{prop.structopti} implies that  $\hat \gamma_{x,t}^+=\hat \gamma_{x,t}^-$ on $(a,t)$. Hence,  we conclude that 
$$
 \partial_A u^{\bar A}(\beta)(x,t)= -\int_0^t \beta(s) {\bf 1}_{\{\hat \gamma_{x,t}^+(s)=0\}} ds,
 $$
 and we can complete the proof as for  Theorem \ref{lem.OC}. 

\end{proof}

We conclude this section showing that, under suitable conditions, $\hat \gamma^+_{x,t}=\hat \gamma^-_{x,t}$ unless two connected components of $\{u^{\bar A}(0,\cdot)<u^{A_0}(0,\cdot)\}$ touch. 

\begin{cor}  Assume \eqref{takis1}, \eqref{takis2} and \eqref{hypphif},  let $\bar A$ be a minimizer  for $\mathcal J$ defined by \eqref{costf2}, and  suppose that $a_1<b_1< a_0<b_0$  are such that $(a_1,b_1)$ and $(a_0,b_0)$ are connected components of $\{u^{\bar A}(0,\cdot)<u^{A_0}(0,\cdot)\}$ with $u^{\bar A}(0,\cdot)=u^{A_0}(0,\cdot)$ in $[b_1,a_0]$. Assume, in addition that, for any $t\in (0,T]$, any optimal trajectory for $u^{A_0}(0,t)$ is positive on $[0,t)$. Finally, we suppose that $A_0= \min_p H^R(0)$. 

Then, for any $(x,t)\in \R\times (0,T)$ point of differentiability of $u^{\bar A}$ for which there is $s\in [a_0,b_0)$ with $\hat \gamma^-_{x,t}(s)= 0$, $\hat \gamma^+_{x,t}$ does not vanish on $[0, a_0)$ and $\hat \gamma^+_{x,t}=\hat \gamma^-_{x,t}$ on $[a_0, t]$. 
\end{cor}


\begin{proof} Let $(x,t)\in \R\times (0,T)$  point of differentiability of $u^{\bar A}(x,t)$  for which there is $s\in [a_0,b_0)$ with $\hat \gamma^-_{x,t}(s)= 0$. Then $\hat \gamma^+_{x,t}(s)= 0$. 
\vs
Assume that $\hat \gamma^+_{x,t}$ vanishes at some $\bar s<a_0$. Then $\hat \gamma^+_{x,t}$ vanishes on $[\bar s, a_0]$. Without loss of generality, we can assume that $\bar s \geq b_1$ (otherwise we replace $\bar s$ by $b_1$). 
\vs
Let $\gamma$ be optimal for $u^{A_0}(0,\bar s)$. As $u^{A_0}(0,\bar s)= u^{\bar A}(0,\bar s)$ and $\gamma$ does not vanish on $(0, \bar s)$, $J^{A_0}(\gamma)= J^{\bar A}(\gamma)$ and $\gamma$ is also optimal for $u^{\bar A}(0,\bar s)$.  Then  dynamic programming implies that 
$$
\tilde \gamma(\tau) =\left\{\begin{array}{ll}
\gamma (\tau) & \text{if } \tau \in [0, \bar s],\\[1mm]
\hat \gamma^+_{x,t} (\tau) & \text{if } \tau \in [\bar s, t]
\end{array}\right.
$$
is also optimal for $u^{\bar A}(x,t)$, and, in turn, again 
by dynamic programming, $\tilde \gamma_{|_{[0,a_0]}}$ is optimal for $u^{\bar A}(0,a_0)$. 
\vs
We know from Lemma \ref{lem.barA=A0} that $\bar A= A_0$ on $[\bar s, a_0]$ and, to fix the ideas, we assume here that $\gamma>0$ on $(0, \bar s)$. Then,
\begin{align*}
u^{\bar A}(0, a_0)= J^{\bar A}( \tilde \gamma_{|_{[0, a_0]}}) & = \int_0^{\bar s} L^R(\dot{\tilde \gamma}) -\int_{\bar s}^{a_0} \bar A(s)ds \\
& =    \int_0^{a_0} L^R( \dot{\tilde \gamma}) > u^{A_0}(0,a_0) \geq u^{\bar A}(0, a_0),  
\end{align*}
which is a contradiction. Note that   the third equality above comes from the equality $-\bar A=-A_0= L^R(0)$ on $[\bar s,a_0]$ and the strict inequality from the strict convexity of $L^R$ and the fact that $\tilde \gamma$ is not a straight line. 
\vs
Finally, we know that $\hat \gamma^+_{x,t}=\hat \gamma^-_{x,t}$ until the first time $s_1$ they touch $\{x=0\}$. Since, in view of the structure of minimizers, $\hat \gamma^+_{x,t}$ and $\hat \gamma^-_{x,t}$ vanish on an interval, they must vanish on $[s,s_1]$. Then they must vanish on $[a_0,s]$ because $u^{\bar A}(0,\cdot)<u^{A_0}(0,\cdot)$ on $(a_0,s)$. This shows that $\hat \gamma^+_{x,t}=\hat \gamma^-_{x,t}$ on $[a_0, t]$. 
\end{proof}

\section{Maximizing the flux} \label{sec.bangbang}

We analyze here the second example in the previous section which is about optimizing the flux on the exit line which is formulated as follows.
\vs

Fix $\xi \in C^1(\R\times (0,T))$ with compact support and, for $A: [0,T] \to [A_0,0]$, the aim is to minimize the functional 
$$\mathcal J(A)= \int_0^T \int_\R \xi(x,t) f(\rho^A)(x,t) dx dt,$$
which, in view of the relationship between the conservation law and the Hamilton-Jacobi equation,  is rewritten as 
\begin{align*}
\mathcal J(A) & =  \int_0^T \int_\R \xi(x,t) f(\rho^A(c,t)) dx dt =   \inf_A \int_0^T \int_{\R} \xi(x,t) (-H(u^A_x(x,t)) )dx dt \\
& 
=  \int_0^T \int_{\R} \xi(x,t) u^A_t(x,t) dx dt = \int_0^T \int_{\R} (-\xi_t(x,t)) u^A(x,t). 
\end{align*} 
\vs
Here we want to increase the traffic in the space-time rectangle $(x_1,x_2)\times (t_1,t_4)$, where $x_1>0$ and $0<t_1<t_4<T$. For this, we choose $\xi(x,t)=-\psi_1(x)\psi_2(t)$ to be an approximation of $-{\bf 1}_{(x_1,x_2)\times (t_1,t_4)}$, so that 
$$
\phi(x,t)=-\xi_t(x,t)= \phi(x,t)= \psi_1(x)\psi_2'(t), 
$$
with  $\psi_1:[0,\infty)\to [0,\infty)$ continuous,  vanishing outside $(x_1, x_2)$ and positive in $(x_1, x_2)$. The $C^1-$map $\psi_2:[0,T]\to \R_+$ is such that $\psi_2'$ vanishes outside the two intervals $(t_1,t_2)$ and $(t_3,t_4)$, is positive on $(t_1,t_2)$ and negative on $(t_3,t_4)$. 
\vs
We summarize all the above in 
\be\label{hypdebasephi}
\begin{split}
\phi:\R\times (0,T) \ & \text{is continuous and compactly supported in}  \ (x_1,x_2)\times (t_1,t_4) \ \text{and}\\
& \{\phi>0\}= (x_1,x_2) \times (t_1,t_2)  \ \text{and} \  \{\phi<0\} = (x_1, x_2) \times(t_3,t_4),
\end{split}
\ee
and write 
\be\label{takis130}
\mathcal J(A)=\int_0^T \int_{\R} \phi(x,t) u^A(x,t) dx dt.
\ee

We will also assume  that 
\be\label{condstrange}
\frac{x_1}{t_2} >\frac{x_2}{t_3}.
\ee
\vs
Note that this condition, which is technical,  holds if $t_3$ is large enough or if $x_1$ and $x_2$ are close. As we will see below, \eqref{condstrange} ensures that no characteristic can cross both $\{\phi>0\}$ and $\{\phi<0\}$. 
\vs

The goal is to show  to show the existence of an optimal control $\bar A$ which is bang-bang, that is,  it only takes its values in $0$ and $A_0$. 
\vs
In order to prove that this optimal control is piecewise constant, that is, the set $\{\bar A=0\}$ consists only of  a finite number of intervals, we need to add three more technical conditions, which we explain next. 
\vs
The first, which is often used in the literature,   is that  $H^R$ and $H^L$ have the same minimum value, that is, 
\be\label{hyp:minHRHL}
\min H^R =\min H^L=A_0,
\ee
which  implies that $L^R(0)=L^L(0)=-A_0$. 
\vs
The second one is  about the  behavior of the  optimal trajectories for $u^{A_0}$. We assume that
\be\label{hyp:optiuA0}
\text{for any $t\in (0,T)$, $\gamma \equiv 0$ is not optimal for $u^{A_0}(0,t)$.}
\ee
Assuming \eqref{hyp:minHRHL}, it is  easily checked  that  a sufficient condition for \eqref{hyp:optiuA0} to hold 
 is that 
\[
\begin{split}
&\text{ either $u_0$ has a right derivative $u_{0,x}(0+)$ at $0$ and $(H^R)'(u_{0,x}(0+))<0$}\\
&\text {or $u_0$ has a left derivative $u_{0,x}(0-)$ at $0$ and  $(H^L)'(u_{0,x}(0-))>0$.}
\end{split}
\]

The last assumption is on the initial condition $\rho_0=-u_{0,x}$. It says that $\rho_0$ cannot oscillate too much between $0$ and positive values, that is,
\be\label{hyp:condu0}
\{x\leq 0 \; : u_{0,x}(x)=0\} \;\text{consists a.e. of only finitely many intervals}.
\ee
\vs
Our main result is the following: 

\begin{thm}\label{thm.mainexTOT} Assume \eqref{takis1}, \eqref{takis2}, \eqref{hypdebasephi},  \eqref{condstrange}, \eqref{hyp:minHRHL}, \eqref{hyp:optiuA0} and  \eqref{hyp:condu0} and, for any measurable $A:(0,T)\to \R$, let $\mathcal J(A)$ be given by \eqref{takis130}. Then,   there exists a minimizer $\bar A$ which takes  values in $\{A_0,0\}$ only and the set $\{\bar A=0\}$ consists a.e. of  at most a finite number of intervals. 
\end{thm}

\begin{rmk} We explain in Remark \ref{rmk:casparticulier} below that, if in addition, 
$$
u_{0,x}>0 \; \text{a.e. in } (-\infty,0)\qquad \text{and} \qquad u_{0,x} < R^R  \; \text{a.e. in } (0,\infty), 
$$
then the set $\{\bar A=0\}$ is made of at most two intervals. 
\end{rmk}

The proof of the theorem is given at the end of the ongoing section. We first show that, under less restrictive conditions, there exists  a bang-bang optimal control $\bar A$. However, this control has possibly  a countable  number of discontinuities. To prove that the number of discontinuities is finite requires some extra work. 
\vs
For the proof, we  start investigating  the behavior of the optimal solution starting from the last connected component of $\{u^{\bar A}(0, \cdot)<u^{A_0}(0, \cdot)\}$, since it controls the  structure of the whole problem. Then we obtain the finiteness result by refining the analysis of the connected components. 

\subsection{Existence of bang-bang optimal control} 

We begin with the result about the bang-bang minimizing strategy which requires fewer assumptions.

\begin{thm}\label{thm.mainexBIS} Assume \eqref{takis1}, \eqref{takis2}, \eqref{hypdebasephi}, \eqref{condstrange} and, for any measurable $A:(0,T)\to \R$, let $\mathcal J(A)$ be given by \eqref{takis130}. If $\min H^R =A_0$, then there exists a minimizer   $\bar A$ which takes only values in $\{A_0,0\}$. In addition, the set $\{\bar A=0\}$ consists a.e. of  at most a countable  number of intervals. 
\end{thm}

Throughout this part we work under the assumption of Theorem \ref{thm.mainexBIS}. The proof, which  requires several steps,   is given at the end of the subsection. 
\vs

Since, in principle, there are more than one minimizing controls, it is necessary to introduce some   further properties of the minimizers, which will allow us to distinguish among them.

Let
$$
\mathcal A= \{ \bar A \; \text{minimizer of $\mathcal J$}\}, \qquad 
\mathcal A_1 =\{ \bar A\in \mathcal A \; : \int_0^T \bar A= \min_{A\in \mathcal A} \int_0^TA\}
$$
and 
$$
\mathcal A_2 = \{ \bar A\in \mathcal A_1\; :\int_0^T s\bar A(s)ds = \min_{A\in \mathcal A_1} \int_0^TsA(s)ds\}
$$
\vs
Requiring that $\bar A$ belongs to $\mathcal A_1$ is quite natural as it was explained in Remark \ref{rem.A=A0inuA=uA0}. The condition that $\bar A$ belongs to $\mathcal A_2$ is mostly technical but plays a key role to prove that the control is bang-bang. 
\vs
From now on we fix $\bar A$ in $\mathcal A_2$. As $\bar A$ belongs to $\mathcal A_1$, we know from Lemma \ref{lem.barA=A0} that $\bar A\equiv A_0$ on $\{u^{\bar A}(0,\cdot)= u^{A_0}(0,\cdot)\}$. It remains to understand what happens in the set $\{u^{\bar A}(0,\cdot)< u^{A_0}(0,\cdot)\}$.
\vs

The first step in this direction is the next lemma.

\begin{lem} \label{lem.barA=A0sincetauBIS} Let 
\be\label{deftautauBIS}
\tau=\sup\left\{ s\in [t_1,t_2] :
\begin{array}{l} 
\text{there exists} \; (x,t)\in (x_1,x_2)\times (s,t_2), \; \\
u^{\bar A}\; \text{is differentiable at $(x,t)$,}
\\
\text{$u^{\bar A}(x,t)<u^{A_0}(x,t)$ and $\hat \gamma^-_{x,t}(s)=0$}\end{array}\right\},
\ee
with $\tau = 0$, if the right-hand side is empty. Then $\bar A\equiv A_0$ a.e. in $(\tau, T)$. 
\end{lem}

\begin{proof} Let 
$$
\tilde A(t)=\left\{\begin{array}{ll}
 A_0 & \text{on $(\tau,T)$,}\\[1mm]
 \bar A(t) & \text{otherwise.}
\end{array}\right. 
$$
Since $A_0\leq \tilde A\leq \bar A$ and $\tilde A=A_0$ in $[0,\tau]$, it follows that  $u^{\bar A}\leq u^{\tilde A}\leq u^{A_0}$ with $u^{\tilde A}= u^{\bar A}$  in $\{u^{\bar A}= u^{A_0}\}$ and in $[0,\infty)\times [0, \tau]$.
 \vs
 Let now $(x,t)\in (0,\infty)\times (\tau,T)$ be a point of differentiability of $u^{\bar A}$ such that $u^{\bar A}(x,t)< u^{A_0}(x,t)$. If $\phi(x,t)<0$, then $\phi(x,t)u^{\bar A}(x,t)\geq \phi(x,t)u^{\tilde A}(x,t)$. If $\phi(x,t)>0$, then $x\in (x_1, x_2)$ and thus, by the definition of $\tau$, $\{\hat \gamma^-_{x,t}=0\}\subset [0,\tau]$. Recall that $\tilde A=\bar A$ on $(0,\tau)$, so that 
\begin{align*}
u^{\bar A}(x,t)& =J^{\bar A}(\hat \gamma^-_{x,t})= \int_0^t (L^R(\dot{\hat \gamma}^-_{x,t}){\bf 1}_{\{\hat \gamma^-_{x,t}> 0\}} +
L^L(\dot{\hat \gamma}^-_{x,t}){\bf 1}_{\{\hat \gamma^-_{x,t}< 0\}} - \bar A{\bf 1}_{\{\hat \gamma^-_{x,t}=0\}})ds\\
& = 
\int_0^t (L^R(\dot{\hat \gamma}^-_{x,t}){\bf 1}_{\{\hat \gamma^-_{x,t}> 0\}} +
L^L(\dot{\hat \gamma}^-_{x,t}){\bf 1}_{\{\hat \gamma^-_{x,t}< 0\}}  - \tilde A{\bf 1}_{\{\hat \gamma^-_{x,t}=0\}})ds= J^{\tilde A}(\hat \gamma^-_{x,t})\\
& \geq u^{\tilde A}(x,t)\geq u^{\bar A}(x,t). 
\end{align*}
This proves that $u^{\tilde A}(x,t)=u^{\bar A}(x,t)$ in this case. 
\vs
In conclusion,  $\phi(x,t)u^{\tilde A}(x,t)\leq \phi(x,t)u^{\bar A}(x,t)$  and, therefore, $\mathcal J(\tilde A)\leq \mathcal J(\bar A)$. Thus $\tilde A$ is also optimal and, as $\tilde A\leq \bar A$, we obtain $\tilde A=\bar A$ since $\bar A$ belongs to $\mathcal A_1$. So $\bar A=A_0$ on $[\tau, T]$. 

\end{proof}

From now on, we implicitly assume that $\tau$ is positive, as Theorems \ref{thm.mainexTOT} and \ref{thm.mainexBIS} obviously hold if $\tau =0$ since then $\bar A\equiv A_0$. 
\vs
The next lemma explains the role of  \eqref{condstrange}. 

\begin{lem}\label{lem:condstrangeBIS}  Let $(x,t)\in [x_1,x_2]\times [t_3,t_4]$ and $\gamma$ be optimal for $u^{\bar A}(x,t)$. If $\gamma(h)=0$  for some $h\in [0,t]$, then $\gamma([h,t])\cap \{\phi>0\}=\emptyset$. Moreover, 
$$
\sup\{s \; : \gamma(s)=0\} \geq \tau.
$$ 
\end{lem}

\begin{proof} In view of the structure of optimal trajectories established in Proposition \ref{prop.structopti}, 
there exists $0\leq s_1<s_2<t$ such that  $\gamma$ is a straight line on $[s_2,t]$, vanishes on $[s_1,s_2]$ and is a nonvanishing straight line on $[0,s_1)$. Therefore $h\in [s_1,s_2]$.
\vs
In order to check that  $\gamma([h,t])\cap \{\phi>0\}= \emptyset$, it is enough to show that, if $s\in [h,t]\cap (t_1,t_2)$, then  $\gamma(s) \leq  x_1$. Let $s\in [h,t]\cap (t_1,t_2)$. Then, it follows from the convexity of the map $s'\to  \gamma(s')$ on $[s_1, T]$ and the facts that $\gamma(h)=0$ and $\gamma(t)=x$, that 
$$
\gamma(s) \leq (s-h)x/(t-h) \leq (s-h)x_2/(t-h) \leq sx_2/t \leq t_2x_2/t_3. 
$$
Hence, by assumption \eqref{condstrange},  
$$
\gamma(s) \leq t_2x_2/ t_3< x_1,
$$
which shows that $\gamma([h,t])\cap \{\phi>0\}=\emptyset$.
\vs
Let us now check that $s_2= \sup\{s \; : \gamma(s)=0\}\geq \tau$. If not, in view of the definition of $\tau$, there exist $(x_n,t_n)\in (x_1,x_2)\times (t_1,t_2)$ such that  $\hat \gamma^-_{x_n,t_n}$ vanishes at $\tau_n$ and $\tau_n\to \tau$. Moreover, we can  assume, without loss of generality, that $\hat \gamma^-_{x_n,t_n}$ is a straight line on $[\tau_n, t_n]$. 
\vs
Since $s_2<\tau$,  it follows  that   $s_2< \tau_n$ for $n$ large. Moreover, we know that  $0=\hat \gamma^-_{x_n,t_n}(\tau_n) < \gamma(\tau_n)$.  Finally, the first part of the proof applied to $h=s_2$ and $s=t_n\in [s_2,t]\cap (t_1,t_2)$ yields $\hat \gamma^-_{x_n,t_n}(t_n)=x_n  \geq x_1 > \gamma(t_n)$. 
Thus, we infer that there exists $s\in (\tau_n, t_n)$ such that $\hat \gamma^-_{x_n,t_n}(s) = \gamma(s)$. 
\vs
But two optimal trajectories cannot cross in $(0,\infty)\times (0,T)$. Thus  there is a contradiction and we must have $s_2\geq \tau$.  

\end{proof}

The next lemma gives information about the location of the left point of any 
connected component of $\{u^{\bar A}(0, \cdot)<u^{A_0}(0,\cdot)\}$ in terms of $\tau$.

\begin{lem}\label{lemalestauBIS} Let $(a,b)$ be any connected component of $\{u^{\bar A}(0, \cdot)<u^{A_0}(0,\cdot)\}$. Then  $a< \tau$. 
\end{lem}

\begin{proof} If $a\geq \tau$, it follows from Lemma~\ref{lem.barA=A0sincetauBIS} that $\bar A\equiv A_0$ on $[\tau,T]$ and, thus,  on $(a,b)$. 
\vs
Fix $\bar t\in (a,b)$ and let $\gamma$ be optimal for $u^{\bar A}(0,\bar t)$. Then from the structure of optimizers, $\gamma$ has to vanish on $[a,\bar t]$, while dynamic programming gives
\begin{align*}
u^{\bar A}(0,\bar t) = -\int_a^{\bar t} \bar A ds +u^{\bar A}(0,a) =  -\int_a^{\bar t}  A_0 ds +u^{ A_0}(0,a) \geq u^{A_0}(0,\bar t). 
\end{align*}
\vs
Then we must have $u^{\bar A}(0,\bar t)= u^{A_0}(0,\bar t)$, a contradiction  to the fact that  $(a,b)$ is a subset of $\{u^{\bar A}(0, \cdot)<u^{A_0}(0,\cdot)\}$.

\end{proof}

In view of the previous lemma, we have either $\tau \in (a,b]$ or $b< \tau$. We continue with the analysis of the later case. 

\begin{lem}\label{lem.cdBIS} Let $(a,b)$ be a connected component of $\{u^{\bar A}(0,\cdot)<u^{A_0}(0,\cdot)\}$ with $b< \tau$. Then there exists $\hat \tau\in (a,b)$ such that $\bar A\equiv 0$ on $(a,\hat \tau)$ and $\bar A\equiv A_0$ on $(\hat \tau, b)$. 
\end{lem}

\begin{proof} We first check that $\bar A\not \equiv A_0$ in $(a,b)$. Indeed, otherwise, for  $t\in (a,b)$,  let $\gamma$ be optimal for $u^{\bar A}(0,t)$. Then $\gamma(s)= 0$ on $[a,t]$ and, thus,  by dynamic programming we must have 
\be\label{oulkqejsfdg}
u^{\bar A}(0,t)= -(t-a)A_0 + u^{\bar A}(0,a)= -(t-a)A_0 + u^{A_0}(0,a), 
\ee
and, after letting $t\to b^-$, 
$$
u^{A_0}(0,b)= u^{\bar A}(0,b)=  -(b-a)A_0 + u^{A_0}(0,a).
$$
Assume to fix the ideas that $A_0= \min H^R$--the other case can be treated in a symmetric way. 
\vs
Then, since  $u^{A_0}_t+H^R(u^{A_0}_x)=0$ in $(0,\infty)\times (0,T)$  with $H^R\geq A_0$, we get that  $u^{A_0}_t\leq -A_0$ a.e. on $(0,\infty)\times (0,T)$. It follows  that $u^{A_0}(0,b)-u^{A_0}(0,a) \leq -A_0(b-a)$. Therefore $u^{A_0}_t(0,t)= -A_0$ on $(a,b)$, which, recalling \eqref{oulkqejsfdg}, implies that $u^{A_0}=u^{\bar A}$ on $\{0\}\times [a,b]$, a  contradiction to  the definition of $(a,b)$. 
\vs
Next we  check that $\bar A\not \equiv 0$ in $(a,b)$. Indeed, otherwise, we have by dynamic programming that $u^{\bar A}(0,\cdot)$ is constant on $[a,b]$ and, thus,  that 
$$u^{A_0}(0,a)=u^{\bar A}(0,a)=u^{\bar A}(0,b)=u^{A_0}(0,b).$$
Since  $u^{A_0}_t+H^R(u^{A_0}_x)=0$ in $(0,\infty)\times (0,T)$ with $H^R\leq 0$,it follows  that $u^{A_0}(0, \cdot)$ is nondecreasing and thus constant in $[a,b]$. So $u^{A_0}=u^{\bar A}$ on $\{0\}\times [a,b]$, which contradicts the definition of $(a,b)$. 
\vs
We now prove the existence of  $\hat \tau\in (a,b)$ such that $\bar A\equiv 0$ on $(a,\hat \tau)$ and $\bar A\equiv A_0$ on $(\hat \tau, b)$. Arguing by contradiction,  we have that, for any $\ep>0$ small, there exists $t_0\in (a,b)$ with $t_0+\ep\in (a,b)$ such that $\bar A\not \equiv 0$ and $\bar A\not \equiv A_0$ a.e. in $(t_0, t_0+\ep)$ and $s\to \bar A (s)$ not (essentially) decreasing  in $(t_0, t_0+\ep)$. Then, we must have 
$$\ep^{-1}\int_{t_0}^{t_0+\ep} \bar A(s)ds \in (A_0,0).$$
\vs
Let 
$$
 A_\ep(s)= \left\{ \begin{array}{ll}
 \bar A(s) &  \text{on}\; (0,T)\backslash (t_0,t_0+\ep),\\[1mm]
0 & \text{on}\; (t_0,  \tau_\ep),\\[1mm]
A_0 & \text{on}\; ( \tau_\ep ,t_0+\ep),
\end{array}\right.
$$
where $ \tau_\ep\in (t_0, t_0+\ep)$ is such that
\be\label{ajzedrnfjk0}
\int_{t_0}^{t_0+\ep}  \bar A(s)ds = \int_{t_0}^{t_0+\ep}   A_\ep(s)ds. 
\ee
Note that, by definition, $A_\ep$ is the unique minimizer, among all measurable maps $x:[t_0,t_0+\ep]\to [A_0,0]$ such that $\int_{t_0}^{t_0+\ep}  \bar A(s)ds= \int_{t_0}^{t_0+\ep} x(s)ds$,  of the map $x\to \int_{t_0}^{t_0+\ep} s x(s)ds$.
\vs
Moreover, for all $t\in (t_0,t_0+\ep)$, 
\be\label{ajzedrnfjk}
\int_{t_0}^t A_\ep(s)ds \geq \int_{t_0}^t \bar A(s)ds. 
\ee
The stability of minimizers, Lemma \ref{lem.cdBIS} and the facts that $b<\tau$ and $A_\ep\to \bar A$ in $L^1$ as $\ep\to 0^+$ yield  some small  $\ep>0$  such that, for any $(x,t)\in \{\phi<0\}$, any optimal solution $\gamma$ for $u^{A_\ep}(x,t)$ which vanishes satisfies $\sup\{s\; : \gamma(s)=0\}\geq b$.  
\vs
We claim that $u^{A_\ep}\leq u^{\bar A}$ in $[0,\infty)\times [0,T]$. Indeed, let $(x,t)\in (0,\infty)\times (0,T)$ be such that $u^{\bar A}$ is differentiable at $(x,t)$. If $u^{\bar A}(x,t)=u^{A_0}(x,t)$, then $u^{A_\ep}(x,t)\leq u^{A_0}(x,t)=u^{\bar A}(x,t)$. 
\vs
Next  assume that $u^{\bar A}(x,t)<u^{A_0}(x,t)$ and let $\gamma=\hat \gamma^-_{x,t}$ be the optimal solution for $u^{\bar A}(x,t)$ which remains the least in $\{x=0\}$. Then, by the structure of minimizers,  there exists $0\leq \tau^-<\tau^+\leq t$ such that $\gamma$ vanishes on $[\tau^-,\tau^+]$ and is affine and does not vanish in $(0, \tau^-)$ and in $(\tau^-,t)$ and, in addition,  $u^{\bar A}(0,\tau^+)<u^{A_0}(0,\tau^+)$ and $u^{\bar A}(0,\tau^-)=u^{A_0}(0,\tau^-)$.
\vs

If $(t_0,t_0+\ep)\cap (\tau^-,\tau^+)=\emptyset$, then  
$$
u^{A_\ep}(x,t) \leq J^{A_\ep}(\gamma) = J^{\bar A}(\gamma) = u^{\bar A}(x,t),
$$
where the first equality holds because $A_\ep=\bar A$ outside $(t_0,t_0+\ep)$ and $\gamma$ vanishes only outside this interval. 
\vs

If $(t_0,t_0+\ep)\cap (\tau^-,\tau^+)\neq \emptyset$, then, since  $u^{\bar A}=u^{A_0}$ in $(0,a)$ and $(0,b)$ and $u^{\bar A}<u^{A_0}$ in $\{0\}\times (a,b)$ and $\gamma = \hat \gamma^-_{x,t}$, we have $\tau^+\in (t_0,b)$ and $\tau^-=a$. The dynamic programming then yields 
\begin{align*}
u^{A_\ep}(x,t) & \leq u^{A_\ep}(0,a)+\int_{\tau^+}^t L^R(\dot\gamma(s))ds -\int_a^{\tau^+} A_\ep(s)ds  \\ 
& \leq u^{\bar A}(0,a) +\int_{\tau^+}^t L^R(\dot\gamma(s))ds -\int_a^{\tau^+} \bar A(s)ds  = u^{\bar A}(x,t),
\end{align*}
where the second inequality holds because  $u^{\bar A}= u^{A_\ep}$ in $\R\times [0,t_0]$,
and  \eqref{ajzedrnfjk} holds. It follows that $u^{A_\ep}\leq u^{\bar A}$. 
\vs
Next we claim that $\{u^{A_\ep}< u^{\bar A}\}\subset \{\phi\geq 0\}$. Let $\gamma$ be optimal for $u^{A_\ep}(x,t)$, where $(x,t)$ is a point of differentiability of $u^{A_\ep}$ and $\phi(x,t)<0$. By the choice of $\ep$, we know that $h=\sup\{s \; :\gamma(s)=0\} \geq b$. 
\vs
If $\{\gamma=0\}\cap (a,b)=\emptyset$, then
$$
u^{A_\ep}(x,t)=J^{A_\ep}(\gamma)= J^{\bar A}(\gamma)\geq u^{\bar A}(x,t),
$$
and, hence $u^{A_\ep}(x,t)=u^{\bar A}(x,t)$. 
\vs
Otherwise $\gamma\equiv 0$ on $(a,b)$ because $u^{A_\ep}\leq u^{\bar A}<u^{A_0}$ on $\{0\}\times (a,b)$ and $\gamma(h)=0$ with $h\geq b$, and,  then 
\begin{align*}
u^{A_\ep}(x,t)& = u^{A_\ep}(0,a) -\int_a^h A_\ep(s)ds + \int_h^t L^R(\dot \gamma)  \\
& = u^{\bar A}(0,a) -\int_a^h \bar A(s)ds + \int_h^t L^R(\dot \gamma) = J^{\bar A}(\gamma)\geq u^{\bar A}(x,t),
\end{align*}
where we used \eqref{ajzedrnfjk0} and the fact that $A_\ep=\bar A$ in $(t_0+\ep, T)$ in the second equality. 
\vs
This shows that $u^{A_\ep}(x,t)=u^{\bar A}(x,t)$ and completes the claim. 
\vs
We now complete the proof of the lemma. Since  $u^{A_\ep}\leq u^{\bar A}$ with  equality in $\{\phi<0\}$, $A_\ep$ is also a minimizer of $\mathcal J$. By construction we have $\int_0^TA_\ep=\int_0^T \bar A$, so that $A_\ep$ belongs to $\mathcal A_1$. Finally, as $A_\ep$ is the unique minimizer, among all measurable maps $x:[t_0,t_0+\ep]\to [A_0,0]$ such that $\int_{t_0}^{t_0+\ep}  \bar A(s)ds= \int_{t_0}^{t_0+\ep} x(s)ds$,   of the map $x\to  \int_{t_0}^{t_0+\ep} s x(s)ds$,
we have 
$\int_0^T s A_\ep(s)ds<\int_0^T s \bar A(s)ds$, which contradicts the assumption that $\bar A\in \mathcal A_2$. 
\vs
It follows that $\bar A$ satisfies the required condition.  
\end{proof}

We now address the case of a connected component $(a,b)$ such that $\tau \in (a,b]$. Note that,  in view of Lemma \ref{lemalestauBIS},  $(a,b)$ must then be the last connected component of $\{u^{\bar A}< u^{A_0}\}$.

\begin{lem}\label{lem.pptetauBIS} Let $\tau$ be defined in  \eqref{deftautauBIS}  and $(a,b)$ be a connected component of $\{u^{\bar A}< u^{A_0}\}$ such that $\tau \in (a,b]$. Then,
$$
\tau = \inf\{s\in [0,T] : \;  \bar A\equiv A_0\; \text{on}\; [s,T]\}.
$$
\end{lem}

\begin{proof}  Let $\hat \tau =   \inf\{s\in [0,T] :  \; \bar A\equiv A_0\; \text{on}\; [s,T]\}$. In view of Lemma \ref{lem.barA=A0sincetauBIS}, we already know that $\hat \tau\leq \tau$. 
\vs

By the definition of $\tau$, there exists $(x_n,t_n)\in (x_1,x_2)\times (t_1,t_2)$ such that  $\hat \gamma^-_{x_n,t_n}$ vanishes at $\tau_n$ and $\tau_n\to \tau$. We can also assume, without loss of generality,  that $\hat \gamma^-_{x_n,t_n}$ is a straight-line on $[\tau_n, t_n]$. 
\vs
If $\hat \tau<\tau$, then  $\tau_n>\hat \tau$ for $n$ large enough. But then, since $\hat \gamma^-_{x_n,t_n}$ vanishes on $(a,\tau_n)$,   
\begin{align*}
u^{\bar A}(x_n,t_n) & = u^{\bar A}(0,\hat \tau\vee a ) + \int_{\tau_n}^{t_n} L^R(\dot{\hat \gamma}^-_{x_n,t_n}) ds -\int_{\hat \tau\vee a }^{\tau_n} A_0 ds\\ 
&=    u^{\bar A}(0,\hat \tau\vee a ) +  \int_{\hat \tau\vee a }^{t_n} L^R(\dot{\hat \gamma}^-_{x_n,t_n}) ds   > u^{\bar A}(0,\hat \tau\vee a )+ \int_{\hat \tau\vee a }^{t_n} L^R(\dot{ \gamma}_n) ds \geq  u^{\bar A}(x_n,t_n) , 
\end{align*}
where $\gamma_n:[\hat \tau, t_n]\to \R$ is the straight line $\gamma_n(s)= s\to (s-\hat \tau\vee a )x_n/(t_n-\hat \tau\vee a )$, the second equality holds because $L^R(0)= -\min H^R =-A_0$, and  the strict inequality comes from the fact that $L^R$ is strictly convex and $\hat \tau\vee a <\tau_n$. 
\vs
Since the conclusion above is impossible, we must  have $\hat \tau=\tau$. 
\end{proof}

In order to proceed,  we need to perturb sightly the optimal control. For this, it is necessary  to have estimates on the behavior of the optimal trajectories of the perturbed problem. This is the aim of the following lemma.

\begin{lem}\label{lem.stabiloBIS} Assume \eqref{takis1} and \eqref{takis2}. For each $\delta>0$, there exists $\ep>0$ so that,  for all measurable $A: [0,T] \to  [A_0,0]$ such that $\|A-\bar A\|_{L^1}\leq \ep$, for all $(x,t)\in (x_1,x_2)\times (t_3,t_4)$ such that $u^A(x,t)<u^{A_0}(x,t)$, and for all $\gamma$ optimal paths for $u^A(x,t)$, 
$$
\sup\{s : \; \gamma(s)=0\}\geq \tau-\delta.
$$
\end{lem}

\begin{proof} 
Set 
\be\label{defmmm}
m= \inf\left\{\gamma(s): \; \begin{array}{l}
\text{there exists} \  (x,t)\in [x_1,x_2]\times [t_3,t_4] \ \text{such that } s\leq t, \\
 \gamma \ \text{is optimal for $u^{\bar A}(x,t)$ and},\\
 \text{ $\gamma$ does not vanish on $[0,t]$}\end{array}\right\}. 
\ee
A standard  compactness argument and Lemma \ref{lem:condstrangeBIS} imply that  $m>0$. Indeed, let $\gamma_n$ and $s_n$ be such that $\gamma_n(s_n)\to m$ as $n\to \infty$  with $\gamma_n$ nonvanishing and optimal for $u^{\bar A}(x_n,t_n)$ for some $(x_n,t_n)\in [x_1,x_2]\times [t_3,t_4]$. Then $\gamma_n$ is a straight-line, so that $\gamma_n(s_n)\geq \min \gamma_n = \min\{x_n, \; \gamma_n(0)\}$. 
\vs
If $m=0$, then $\gamma_n(0)$ must  converge to $0$ since $x_n\geq x_1$. Moreover, up to a subsequence, the $(x_n,t_n)$'s  converge to some $(x,t)$ and the $\gamma_n$'s  converge to a straight line $\gamma$ which is optimal for $u^{\bar A}(x,t)$. Note that $\gamma(0)=0$ and thus $\gamma(s)= sx/t$. So $\gamma$ vanishes, but $\sup\{s : \; \gamma(s)=0\}= 0<\tau$, which contradicts Lemma \ref{lem:condstrangeBIS}.
 \vs

Fix $\delta'=\min\{m/2, x_1\delta/t_4\}$ and let $\ep>0$ be such that, for any measurable $A: [0,T] \to  [A_0,0]$ such that $\|A-\bar A\|_{L^1}\leq \ep$, for any $(x,t)\in (x_1,x_2)\times (t_3,t_4)$ and  any $\gamma$ optimal for $u^A(x,t)$, there exists $\tilde \gamma$ optimal for $u^{\bar A}(x,t)$ such that $\|\gamma-\tilde \gamma\|_{L^\infty} \leq \delta'$. 
\vs
Fix such $A$, $(x,t)\in (x_1,x_2)\times (t_3,t_4)$ and $\gamma$. Assume $u^A(x,t)<u^{A_0}(x,t)$. We know that $\gamma$ vanishes. Set 
$$
s^+=\sup\{s : \; \gamma(s)=0\}.
$$
It follows that  that $\gamma(s)= x(s-s^+)/(t-s^+)$ on $[s^+,t]$. 
\vs
Let $\tilde \gamma$ be optimal for $u^{\bar A}(x,t)$ such that $\|\gamma-\tilde \gamma\|_\infty \leq \delta'$. If $\tilde \gamma$ does not vanish on $[0,t]$, then, in view of the definition of $m$ in \eqref{defmmm},  
$$
\delta' \geq |\tilde \gamma(s^+)-\gamma(s^+)|\geq m,
$$
which contradicts the choice of $\delta'$. Thus $\tilde \gamma$ vanishes and, by Lemma \ref{lem:condstrangeBIS}, 
$$
\tilde s^+= \sup\{s : \; \gamma(s)=0\}\geq \tau. 
$$
If $s^+\leq \tau$, then 
$$
\delta' \geq  |\tilde \gamma(\tilde s^+)-\gamma(\tilde s^+)| \geq x(\tilde s^+-s^+)/(t-s^+) \geq x_1 (\tau-s^+)/t_4. 
$$
Hence 
$$
s^+\geq \tau -t_4\delta'/x_1 \geq \tau - \delta.
$$
\end{proof}

\begin{lem}\label{lem.A=0dansatauBIS} Let $\tau$ defined in  \eqref{deftautauBIS}  and $(a,b)$ be a connected component of $\{u^{\bar A}< u^{A_0}\}$ such that $\tau \in (a,b]$. Then 
$\bar A\equiv 0$ on $(a,\tau)$. 
\end{lem}

\begin{proof}  The argument is very close to the proof of Lemma \ref{lem.cdBIS}.  The key idea is to prove that 
\be\label{takis140}
\bar A  \text{ is essentially nonincreasing on $(a,\tau)$ and takes values in $\{A_0,0\}$ only.} 
\ee
\vs
Assume that \eqref{takis140} holds. By Lemma \ref{lem.pptetauBIS}, there is a sequence of Lebesgue point $(s_n)_{n\in\N}$ for $\bar A$ converging to $\tau^-$ such that $\bar A(s_n)>A_0$. Thus we infer that $\bar A(s_n)=0$ and, as $\bar A$ is essentially nonincreasing, that $\bar A\equiv 0$ on $(a,s_n)$. Letting $n\to \infty$ proves that $\bar A\equiv 0$ in $(a,\tau)$ and completes the proof of the lemma. 
\vs
It remains to verify \eqref{takis140}. 
Fix $\delta>0$ small. Lemma \ref{lem.stabiloBIS} yields an  $\ep>0$ such that, for any measurable $A:[0,T]\to [A_0,0]$ with $\|A-\bar A\|_{L^1}\leq |A_0|\ep$, any optimal trajectory $\gamma$ started from a point  $(x,t)\in (x_1,x_2)\times (t_3,t_4)$ of differentiability of $u^{\bar A}$ with $u^A(x,t)<u^{A_0}(x,t)$ satisfies $\sup\{s: \; \gamma(s)=0\}\geq \tau -\delta$. 
 \vs
Let now $t_0\in (a, \tau-\delta-\ep)$. It needs to be checked that  $\bar A$ is nonincreasing on $(t_0,t_0+\ep)$ and takes values in $\{A_0,0\}$. This is obvious, if  $\bar A \equiv 0$ on $(t_0,t_0+\ep)$ or $\bar A \equiv \bar A$ on  $(t_0,t_0+\ep)$. 
\vs
Assume next  that $\ep^{-1}\int_{t_0}^{t_0+\ep} \bar A\in (A_0,0)$, and let  
$$
 A_\ep(s)= \left\{ \begin{array}{ll}
 \bar A(s) &  \text{on}\; (0,T)\backslash (t_0,t_0+\ep),\\[1mm]
0 & \text{on}\; (t_0,  \tau_\ep),\\[1mm]
A_0 & \text{on}\; ( \tau_\ep ,t_0+\ep),
\end{array}\right.
$$
where $ \tau_\ep\in (t_0, t_0+\ep)$ is such that
\be\label{ajzedrnfjk0TER}
\int_{t_0}^{t_0+\ep}  \bar A(s)ds = \int_{t_0}^{t_0+\ep}   A_\ep(s)ds. 
\ee
We can follow the argument in the proof of Lemma \ref{lem.cdBIS} to infer that $A_\ep$ belongs to $\mathcal A_2$ and, therefore, $A_\ep=\bar A$, and,  thus $\bar A$ is nonincreasing on $(t_0,t_0+\ep)$.  The fact that it takes the values $A_0$ and $0$ now follows.
\end{proof}

We can give now the proof of the main theorem of this subsection.

\begin{proof}[Proof of Theorem \ref{thm.mainexBIS}] Lemma \ref{lem.barA=A0} yields  that $\bar A$ equals $A_0$ on $\{u^{\bar A}(0,\cdot)= u^{A_0}(0,\cdot)\}$, and  let $\tau$ be defined by \eqref{deftautauBIS}. 
\vs
Assume that  $(a,b)$ is the connected component of $\{u^{\bar A}(0,\cdot)<u^{A_0}(0, \cdot)\}$ such that $\tau \in (a, b]$ if it exists.  Then Lemma \ref{lem.barA=A0sincetauBIS} states that $\bar A\equiv A_0$ on $[\tau, T]$ while Lemma \ref{lem.A=0dansatauBIS} shows that $\bar A\equiv 0$ on $(a, \tau)$. 
\vs
If $(a,b)$ is another connected component of $\{u^{\bar A}(0,\cdot)<u^{A_0}(0, \cdot)\}$, then Lemma \ref{lem.cdBIS} shows that $\bar A\equiv 0$ on $(a,\hat \tau)$ and $\bar A\equiv A_0$ on $(\hat \tau,b)$ for some $\hat \tau \in (a,b)$. 
\vs
It follows that $\bar A$ takes values in $\{A_0,0\}$ only. As the number of connected components of $\{u^{\bar A}(0,\cdot)<u^{A_0}(0, \cdot)\}$ is at most countable, the set $\{\bar A=0\}$ consists a.e. in at most a countable number of intervals. 

\end{proof}

\subsection{Optimal trajectories starting from the last connected component} 

The goal of this subsection  is to show that optimal solutions starting from the last connected component of $\{u^{\bar A}(0,\cdot)<u^{A_0}(0, \cdot)\}$ cannot remain for a long time on  $\{x=0\}$. 
\vs

Throughout this part, we work under the  assumption that 
\be\label{takis150}
\text{assume  that \eqref{takis1}, \eqref{takis2}, \eqref{hypdebasephi},  \eqref{condstrange}, \eqref{hyp:minHRHL} and \eqref{hyp:optiuA0} hold.}
\ee

The following lemma  explains the role of assumption \eqref{hyp:optiuA0}. 

\begin{lem}\label{lem.hyp:optiuA0}  Assume \eqref{takis150}. Let $\gamma$ be optimal for $u^{A_0}(0,t)$ for  $t\in (0,T)$. Then $\gamma$ is a  straight line which does not vanish on $[0,t)$.
\end{lem}

\begin{proof}  We argue by contradiction. If $\gamma$ vanishes in $[0,t)$, then, in view of the structure of optimal trajectories  described in Proposition \ref{prop.structopti}, there exists $a\in [0,t)$ such that $\gamma$ vanishes on $[a,t]$ and, if $a>0$,  is a nonvanishing straight line in $(0,a)$. 
\vs
It follows from  \eqref{hyp:optiuA0}  that $a>0$. 
\vs
Assuming, to fix the ideas, that $\gamma(0)>0$ and using  again dynamic programming we find
\begin{align*}
u^{A_0} (0,t)  & = \int_0^a L^R(\dot \gamma) ds  -\int_a^t A_0 ds  +u_0(\gamma(0)) \\
&= \int_0^t L^R(\dot \gamma) ds +u_0(\gamma(0)) > \int_0^t L^R(\dot{\tilde \gamma}) ds+u_0(\tilde \gamma(0)),
\end{align*}
where $\tilde \gamma$ is the straight line connecting $(\gamma(0),0)$ and $(0,t)$, the second equality holds because $A_0= -L^R(0)$ and the strict inequality is due to the strict convexity  $L^R$. 
\vs
The conclusion  is impossible,  since $\tilde \gamma$ is admissible. 

\end{proof}

We now state  a technical lemma which shows that all connected components of \\
$\{u^{\bar A}(0,\cdot)<u^{A_0}(0, \cdot)\}$  must somehow ``see'' $\{\phi>0\}$. 

\begin{lem}\label{lem:variante52} Assume \eqref{takis150}, fix a   be a connected component of $\{u^{\bar A}(0,\cdot)<u^{A_0}(0, \cdot)\}$  and let 
\be\label{takis170}
\mathcal O=\left\{(x,t) : \; \begin{array}{l}
u^{\bar A}(x,t)<u^{A_0}(x,t) \ \text{and}\\[1mm]
\text{there exists} \ \gamma \; \text{optimal for } u^{\bar A}(x,t) \ \text{such that}\\[1mm]
\sup\{s : \; \gamma(s)\}\in (a,b)\end{array}\right \}.
\ee
Then,  $|\mathcal O\cap \{\phi> 0\}|>0$.
\end{lem}

\begin{proof} Define
 $$
 \tilde A(t)=\left\{\begin{array}{ll}
 A_0 & \text{on $(a,b)$,}\\[1mm]
 \bar A(t) & \text{otherwise.}
\end{array}\right. 
$$
Since $\tilde A\leq \bar A$ we have $u^{\bar A}\leq u^{\tilde A}$. 
We claim that 
\be\label{zliekj:rdfs,BIS}
u^{\tilde A}(x,t) = u^{\bar A}(x,t) \ \   \text{in} \ \  ((0,\infty)\times (0,T))\backslash \mathcal O.
\ee
Fix  $(x,t)\in ((0,\infty)\times (0,T))\backslash \mathcal O$, let $\gamma_{x,t}$ be an optimal path for $u^{\bar A}(x,t)$  and set 
$$
\tau^+=\sup\{s\in [0,t] : \; \gamma_{x,t}(s)=0\}.
$$
In view of the definition of $\mathcal O$,  observe that $\tau^+\notin (a,b)$. 
\vs
If $\gamma_{x,t}(s)\neq 0$ for any $s\in(a,b)$, then 
$$
u^{\bar A}(x,t)= J^{\bar A}(\gamma_{x,t}) = J^{\tilde A}(\gamma_{x,t}) \geq u^{\tilde A}(x,t)\geq u^{\bar A}(x,t),  
$$
which proves \eqref{zliekj:rdfs,BIS} in this case. 
\vs
We assume next  that  $\gamma_{x,t}$ vanishes at a point in $(a,b)$. This implies  that $\tau^+\geq b$ and, thus,  in view of the structure of the minimizers, $\gamma_{x,t}$ vanishes on $[b, \tau^+]$.
\vs
Define 
$$
\tilde \gamma_{x,t}(s)= \left\{\begin{array}{ll}
 \gamma^{A_0}_{0,b}(s) & \text{on  $[0,b]$,}\\[1mm]
\gamma_{x,t}(s) & \text{on $[b,t]$, }
\end{array}\right. 
$$
where $ \gamma^{A_0}_{0,b}$ is optimal for $u^{A_0}(0,b)$, and recall that, in view of  Lemma \ref{lem.hyp:optiuA0},  $ \gamma^{A_0}_{0,b}$ vanishes only at $t=b$. 
\vs
Then by dynamic programming and the optimality of $\gamma_{x,t}$, we find 
\begin{align*}
u^{\bar A}(x,t) &= \int_{b}^t \big[L^R(\dot \gamma_{x,t}){\bf 1}_{\{\gamma_{x,t}(s)> 0\}}-\bar A(s){\bf 1}_{\{\gamma_{x,t}(s)=0\}}\big] ds + u^{\bar A}(0,b)\\
& =  \int_{b}^t \big[L^R(\dot \gamma_{x,t}){\bf 1}_{\{\gamma_{x,t}(s)> 0\}}-\tilde A(s){\bf 1}_{\{\gamma_{x,t}(s)=0\}}\big]ds + u^{A_0}(0,b)\\
& = \int_{0}^t \big[ L^R(\dot{\tilde \gamma}_{x,t}){\bf 1}_{\{\tilde \gamma_{x,t}(s)> 0\}}+  L^L(\dot{\tilde \gamma}_{x,t}){\bf 1}_{\{\tilde \gamma_{x,t}(s)< 0\}}-\tilde A(s){\bf 1}_{\{\tilde \gamma_{x,t}(s)=0\}}\big] ds + u_0(\tilde \gamma_{x,t}(0))\\
& =  J^{\tilde A}(\tilde \gamma_{x,t}) , 
\end{align*}
where the second equality holds because $\bar A=\tilde A$ on $[b, t]$ and $u^{\bar A}(0,b)= u^{A_0}(0,b)$ and the third one because $\gamma^{A_0}_{0,b}$ does not vanish on $[0, b)$. 
\vs
Thus
$$
u^{\bar A}(x,t)  = J^{\tilde A}(\tilde \gamma_{x,t}) \geq u^{\tilde A}(x,t)\geq u^{\bar A}(x,t), 
$$
and, hence,  \eqref{zliekj:rdfs,BIS} is true also in this case.
\vs
Assume that $|\mathcal O\cap \{\phi> 0\}|= 0$. Then, in view of   \eqref{zliekj:rdfs,BIS} and the fact that $u^{\bar A}\leq u^{\tilde A}$ and $\phi\leq 0$ a.e. in $\mathcal O$, we have 
\begin{align*}
\mathcal J(\tilde A)-\mathcal J(\bar A)= \int_{\mathcal O} \phi(x,t) (u^{\tilde A}(x,t)-u^{\bar A}(x,t)) dxdt \leq 0,
\end{align*}
and, thus,  $\mathcal J(\tilde A)\leq \mathcal J(\bar A)$, which proves that $\tilde A$ is also optimal. 
\vs
Moreover, since $\bar A\in \mathcal A_1$, we have $\int_0^T \tilde A \geq \int_0^T\bar A$, which implies that $\bar A\equiv \tilde A$ since $\bar A\geq\tilde A$. It follows  $\bar A=\tilde A=A_0$ in $(a,b)$, which, as we show next,   is not possible. 
\vs
Fix $t\in (a,b)$ and let $\gamma_{0,t}$ and $\gamma^{A_0}_{0,a}$  be optimal  for respectively $u^{\bar A}(0,t)$ and
for $u^{A_0}(0,a)$. We recall that  $ \gamma^{A_0}_{0,a}$ vanishes only at $t=a$ by Lemma \ref{lem.hyp:optiuA0} and define
$$
\tilde \gamma_{x,t}(s)= \left\{\begin{array}{ll}
 \gamma^{A_0}_{0,a}(s) & \text{on  $[0,a]$,}\\[1mm]
\gamma_{0,t}(s) & \text{on $[a,t]$.}
\end{array}\right. 
$$
Assume that $ \gamma^{A_0}_{0,a}(0)>0$. 
Then, as $\gamma_{0,t}$ vanishes on $[a,t]$ by Proposition \ref{prop.structopti} and $\bar A=A_0=-L^R(0)$ in $(a,b)$, 
\begin{align*}
u^{\bar A}(0,t) &= - \int_{a}^t  A_0 ds + u^{\bar A}(0,a) =   \int_{a}^t L^R(\dot \gamma_{x,t}) ds + u^{A_0}(0,a)\\
& = \int_{0}^t L^R(\dot{\tilde \gamma}_{x,t}) ds + u_0(\tilde \gamma_{x,t}(0)) >  \int_{0}^t L^R(\dot\gamma) ds + u_0(\gamma(0)) , 
\end{align*}
where $\gamma$ is the straight-line between $(0,\tilde \gamma_{x,t}(0))$ and $(0,t)$. Note that the second equality holds because $\gamma^{A_0}_{0,a}$ is positive on $[0,a)$ and the strict inequality because $L^R$ is strictly convex. 
\vs
Thus 
$$
u^{\bar A}(0,t) > \int_{0}^t L^R(\dot\gamma) ds + u_0(\gamma(0))= J^{A_0}(\gamma) \geq u^{A_0}(0,t),
$$
which is impossible since $ u^{A_0}\geq u^{\bar A}$. 
\vs
If  $ \gamma^{A_0}_{0,a}(0)<0$, then we argue in the same way with $L^L$ instead of $L^R$ using  that $A_0=-L^L(0)$. 
\vs
The fact that  $|\mathcal O\cap \{\phi> 0\}|>0$ now follows.
\end{proof}

From now on,  we denote by $(\bar a, \bar b)$ the last connected component of $\{u^{\bar A}(0, \cdot)< u^{A_0}(0, \cdot)\}$ if there is one, and recall that $\tau\in (\bar a, \bar b]$, where $\tau$ is defined by \eqref{deftautauBIS}. 
\vs 
Let $\hat \gamma^+_{0, \bar b}$ be the  optimal solution for $u^{\bar A}(0, \bar b)$ which remains the most on $\{x=0\}$ and define 
\be\label{takis160}
\bar s^-=\inf\{s : \; \hat \gamma^+_{0, \bar b}(s)=0\}.
\ee

\begin{lem}\label{lem.hatgammaphineg} Assume \eqref{takis150}. Fix   $(x,t)\in \R\times (0,T)$ and let $\gamma$ be an optimal trajectory for $u^{\bar A}(x,t)$. If $\gamma$ vanishes at some point in $(0,t)$ and $\sup\{s :  \; \gamma(s)=0\}\in [\bar s^-, \bar b]$, then $\{s :  \; \gamma(s)= 0\}\subset [\bar s^-, \bar b]$. In particular, if $(x,t)\in (x_1,x_2)\times (t_3,t_4)$ is a point of differentiability of $u^{\bar A}$ such that $u^{\bar A}(x,t)<u^{A_0}(x,t)$, then 
$$
\sup\{s : \; \hat \gamma^+_{x,t}(s)=0\} \in [\tau,\bar  b) \ \ \text{and} \ \ \inf\{s :  \; \hat \gamma^+_{x,t}(s)=0\} =\bar s^-.
$$
\end{lem}

\begin{proof} Fix $\gamma$ as in the statement 
and assume  that $s^+=\sup\{s : \; \gamma(s)=0\}\in [\bar s^-, \bar b]$. Let $s^-= \inf\{s :  \;  \gamma(s)=0\}$. Then, by the structure of minimizers,  $\{s : \; \gamma(s)= 0\}=[s^-,s^+]$. 
\vs
If we had $s^-<\bar s^-$, then using  dynamic programming, we can  replace $\hat \gamma^+_{0, \bar b}$ by $\gamma$ on $[0, s^+]$ to find an optimal trajectory for $u^{\bar A}(0, \bar b)$ which remains  in $\{x=0\}$ longer  than $\hat \gamma^+_{0, \bar b}$, a contradiction with the definition of $\hat \gamma^+_{0, \bar b}$. Thus $s^-\geq \bar s^-$. 
\vs

Let  $(x,t)\in (x_1,x_2)\times (t_3,t_4)$ be a point of differentiability of $u^{\bar A}$ such that $u^{\bar A}(x,t)<u^{A_0}(x,t)$, and  set $s^+ =\sup\{s :\; \hat \gamma^+_{x,t}(s)=0\}$ and $s^-= \inf\{s : \; \hat \gamma^+_{x,t}(s)=0\}$.
\vs
Proposition \ref{prop.structopti} yields that $u^{\bar A}(0,s^+)<u^{A_0}(0, s^+)$ and  Lemma \ref{lem:condstrangeBIS} implies that $s^+\geq \tau$. Thus, if $(a,b)$ is the connected component of $\{u^{\bar A}(0, \cdot)< u^{A_0}(0, \cdot)\}$ containing $s^+$, we have $b\geq \tau$, so that $(a,b)=(\bar a, \bar b)$. Note that $s^-\leq \bar a$. 
\vs
By the first claim,  we also have that $\{s, \; \hat \gamma^+_{x,t}(s)\} \subset [\bar s^-, \bar b]$. If  $s^->\bar s^-$, we can replace $\hat \gamma^+_{x,t}$ by $\hat \gamma^+_{0,\bar b}$ on $[0,\bar a]$ to find  an optimal solution for $u^{\bar A}(x,t)$ which remains  in $\{x=0\}$ longer than $\hat \gamma^+_{x,t}$, a contradiction to  the definition of $\hat \gamma^+_{x,t}$.  Hence $s^-=\bar s^-$. 

\end{proof}

We investigate next the sign of the quantity 
$$
\mathcal H(s)= \int_s^T \int_0^\infty \phi(x,t) {\bf 1}_{\{\hat \gamma^+_{x,t}(s)=0\}}dxdt,
$$
which appears in the optimality condition.

\begin{lem} \label{lem.HnonincBIS}  Assume \eqref{takis150}. The map $s\to \mathcal H(s)$ is nonincreasing  and nonnegative on $(\bar s^-,\tau)$. 
\end{lem}

\begin{proof} Let $s\in (\bar s^-,\tau)$ and $\delta>0$ small so  that $s+\delta<\tau$. We have 
\begin{align*}
\mathcal H(s)& =\mathcal H(s+\delta) + \int_s^{s+\delta}\int_0^\infty \phi(x,t) {\bf 1}_{\{\hat \gamma^+_{x,t}(s)=0\}}dxdt \\
& + \int_{s+\delta}^T \int_0^\infty
\phi(x,t) ({\bf 1}_{\{\hat \gamma^+_{x,t}(s)=0\}}-{\bf 1}_{\{\hat \gamma^+_{x,t}(s+\delta)=0\}})dxdt \\[1mm]
&= \mathcal H(s+\delta) +(H_1)+(H_2).
\end{align*}
For $(H_1)$, we recall that, for any  $(x,t)\in \R_+\times (s,s+\delta)$ such that  $\phi(x,t)\neq0$, we know that  $\hat \gamma^+_{x,t}(t)=x\geq x_1$, and, thus, if we choose $\delta$ small enough, 
$$
\hat \gamma^+_{x,t}(s)\geq x_1-\|(H^R)'\|_\infty(t-s) \geq x_1-\|(H^R)'\|_\infty\delta>0,
$$
Thus $\hat \gamma^+_{x,t}(s)\neq 0$ and, hence,  $(H_1)=0$. 
\vs
For $(H_2)$, let $(x,t)$ be a point of differentiability of $u^{\bar A}$ with $\phi(x,t)\neq0$, such that either $\hat \gamma^+_{x,t}(s)=0$ or  $\hat \gamma^+_{x,t}(s+\delta)=0$. 
\vs 
If $\phi(x,t)<0$, then, by Lemma \ref{lem.hatgammaphineg},  $\hat \gamma^+_{x,t}$ vanishes at least on $[\bar s^-,\tau]$. Thus $\hat \gamma^+_{x,t}(s)=\hat \gamma^+_{x,t}(s+\delta)=0$ and the integrant of $(H_2)$ vanishes for such $(x,t)$. 
\vs
If $\phi(x,t)>0$, then $\hat \gamma^+_{x,t}$ vanishes on an interval of the form $[\bar s^-,h]$ with $h\geq s$, and, thus, if $\hat \gamma^+_{x,t}(s+\delta)=0$, then $\hat \gamma^+_{x,t}(s)=0$. In this case the integrant of $(H_2)$ at $(x,t)$ is nonnegative. Hence $(H_2)\geq 0$. 
\vs

In view of the above, we have  $\mathcal H(s)\geq \mathcal H(s+\delta)$, that is, the map  $s\to \mathcal H(s)$ is nonincreasing on $(\bar s^-,\tau)$. 
\vs

Now recall that, in view of Lemma~\ref{lem.A=0dansatauBIS},   $\bar A\equiv 0$ on $(\bar a, \tau)$. Since  $(\bar a, \bar b)$ is the last connected component of  $\{u^{\bar A}(0, \cdot)< u^{A_0}(0, \cdot)\}$, Corollary \ref{cor.lastinter} yields that $\mathcal H$ is nonnegative on $(\bar a, \tau)$. As  $s\to \mathcal H(s)$ is nonincreasing on $(\bar s^-,\tau)$, we infer that $\mathcal H$ is nonnegative  on $(\bar s^-, \tau)$. 
\end{proof}

We proceed with our investigation of the last connected component $(\bar a, \bar b)$ in the case where $\hat \gamma^+_{0, \bar b}$ remains in $\{x=0\}$ before $\bar a$.  

\begin{lem}\label{lem.seqtn} Assume \eqref{takis150} and suppose  that $\bar s^-<\bar a$.
Then there exists a sequence $(t_n)_{n\in \N}$ such that, $n\to \infty$,  $t_n\to \bar a^-$ and $u^{\bar A}(0,t_n)<u^{A_0}(x,t_n)$.
\end{lem}

\begin{proof} We argue by contradiction assuming that  there exists $\delta\in(0,\bar a-s^-]$ such that $u^{\bar A}(0,t)=u^{A_0}(0,t)$ for any $t\in [\bar a-\delta, \bar a]$. Then $\bar A\equiv A_0$ in $(\bar a-\delta, \bar a)$. 
\vs
Let $\gamma$ be optimal for $u^{A_0}(0, \bar a-\delta)$. It follows that does $\gamma$ does not vanish on $[0, \bar a-\delta)$, and, to fix the ideas, we assume that it  remains negative,  and  set 
$$
\tilde \gamma (s)= \left\{\begin{array}{ll}
\gamma(s) & \text{if $s\in [0, \bar a-\delta]$,}\\[1mm]
\hat \gamma^+_{0, \bar b}(s) & \text{if $s\in [\bar a-\delta, \bar b]$}.
\end{array}\right.
$$
Then, since  $L^R(0)= -A_0$ and $\hat \gamma_{0,b}^+$ vanishes on $[\bar s^-,b]$,  
\begin{align*}
u^{A_0}(0, \bar b) & = u^{\bar A}(0, \bar b)= -\int_{\bar a-\delta}^{\bar b} A_0ds + u^{\bar A}(0, \bar a-\delta) 
= \int_{\bar a-\delta}^{\bar b} L^R(\dot{\tilde \gamma})ds + u^{A_0}(0, \bar a-\delta)  \\
&= \int_0^{\bar b} L^R(\dot{\tilde \gamma})ds +u_0(\tilde \gamma(0)\; > \; \int_0^{\bar b} L^R(\dot{\check \gamma})ds +u_0(\check \gamma(0))= J^{A_0}(\check \gamma), 
\end{align*}
where $\check \gamma$ is the straight-line between $(\gamma(0),0)$ and $(0, \bar b)$ and  the strict inequality follows from the strong convexity of $L^R$. This is impossible because $\check \gamma$ is admissible.
\end{proof}

The next  lemma states that, if the optimal trajectory starting from $(0,\bar b)$ vanishes outside $[\bar a, \bar b]$, then there exists a one before last connected component of $(\hat a, \hat b):=(\bar s^-,\bar a)$ which touches the last one at $\bar a$.

\begin{lem} \label{lem.onebeforelast} If $\bar s^-<\bar a$,
then $(\bar s^-,\bar a)$ is a connected component of $\{u^{\bar A}(0,\cdot)<u^{A_0}(0,\cdot)\}$.
\end{lem}


\begin{proof} Let $(t_n)$ be the sequence introduced in Lemma \ref{lem.seqtn} and $(a_n,b_n)$ be the connected component of $\{u^{\bar A}(0,\cdot)<u^{A_0}(0,\cdot)\}$ containing $t_n$. Then $b_n\leq \bar a$. 
\vs
Note also that, since  $\bar s^-=\inf\{s:  \; \hat \gamma^+_{0,\bar b}(s)=0\}$, Proposition \ref{prop.structopti}  yields that we have $u^{\bar A}(0, \bar s^-)=u^{A_0}(0, \bar s^-)$, and, since $u^{\bar A}(0, \cdot)<u^{A_0}(0, \cdot)$ in $(a_n,b_n)$, we have 
$a_n\geq \bar s^-$.
\vs
We claim that $\mathcal H$ is positive in a neighborhood of $a_n$. Indeed,  let $\delta$ be as in the proof of Lemma~\ref{lem.HnonincBIS}, that is, assume that  $\delta <x_1/\|(H^R)'\|_\infty$, and, consider the sets 
$$
\mathcal O_n=\left\{(x,t) : \ \begin{array}{l}
u^{\bar A}(x,t)<u^{A_0}(x,t) \ \text{and}\\
\text{there exists} \ \gamma \  \text{optimal for } u^{\bar A}(x,t) \ \text{and}\\
 \sup\{s: \; \gamma(s)=0\} \in (a_n,b_n) \end{array}\right\}.
$$
Recall that  it follows from Lemma \ref{lem:variante52} that $\{\phi>0\}\cap \mathcal O_n$ has a positive measure. Thus, there exists $s_n\in (a_n,b_n)$ such that the sets 
$$
E_n= \{(x,t)\in \{\phi>0\}\cap \mathcal O_n:  \  \sup\{s, \; \gamma^+_{x,t}(s)=0\}\in (s_n-\delta/2, s_n)\}
$$
has a positive measure. Choosing, if necessary,  $\delta$ smaller, we can assume that \\
$(s_n-\delta/2, s_n+\delta/2)\subset (a_n,b_n)$. 
\vs
Then, arguing as in the proof of  Lemma \ref{lem.HnonincBIS},  we have
\begin{align*}
\mathcal H(s_n-\delta/2)& =\mathcal H(s_n+\delta/2) + \int_{s_n-\delta/2}^{s_n+\delta/2}\int_0^\infty \phi(x,t) {\bf 1}_{\{\hat \gamma^+_{x,t}(s_n-\delta/2)=0\}}dxdt \\
& + \int_{s_n+\delta/2}^T \int_0^\infty
\phi(x,t) ({\bf 1}_{\{\hat \gamma^+_{x,t}(s_n-\delta/2)=0\}}-{\bf 1}_{\{\hat \gamma^+_{x,t}(s_n+\delta/2)=0\}})dxdt \\
&= \mathcal H(s+\delta) +(H_1)+(H_2),
\end{align*}
where $(H_1)$  vanishes by the choice of $\delta$. We also know from the proof  of Lemma \ref{lem.HnonincBIS} that the integrand of $(H_2)$ is nonnegative, and, by construction,  it is positive on $E_n$. 
\vs
Since $\mathcal H(s_n+\delta/2)$ is nonnegative, it follows that  $\mathcal H(s_n-\delta/2)$ is positive and, therefore, that $\mathcal H$ is positive on $(\bar s^-, s_n-\delta/2)$. Since $\delta$ is arbitrary, this proves that $\mathcal H$ is positive on $(\bar s^-, s_n)$, and, thus, on $(\bar s^-, a_n)$. 
\vs

It follows from Theorem \ref{lem.OC} and the positivity of  $\mathcal H$ on $(\bar s^-,s_n)$, that $\bar A\equiv 0$ on this interval. 
\vs
Recall that, in view of  Lemma \ref{lem.cdBIS}, for any connected component $(a,b)$ of $\{u^{\bar A}(0,\cdot)<u^{A_0}(0,\cdot)\}$ which is not the last one, there exists a threshold $\tau\in (a,b)$ such that $\bar A\equiv 0$ on $(a,\tau)$ and $\bar A\equiv A_0$ in $(\tau, b)$. We denote by $\tau_n\in (a_n,b_n)$ the threshold associated with $(a_n,b_n)$.  
\vs

We now show that $b_n= \bar a$. Indeed, otherwise there would be $n'>n$ such that $t_{n'}>b_n$. But the argument above shows that $\bar A\equiv 0$ on $(\bar s^-,a_{n'})$ with $a_{n'}\geq b_n$, a  contradiction to  the fact that $\bar A\equiv A_0$ in $(\tau_n, b_n)$.
\vs

We now claim that $a_n= \bar s^-$. Indeed,  assuming that $\bar s^-<a_n$, we note that, in view of the fact that  $\bar A\equiv 0$ in $(\bar s^-,a_n)$, $u^{\bar A}(0, \cdot)$ is constant in this interval. But $u^{\bar A}(0,a_n)= u^{A_0}(0,a_n)$ and $u^{A_0}(0, \cdot)$ is nondecreasing, so that, for any $t\in [\bar s^-,a_n)$, 
$$
u^{\bar A}(0,t)\leq u^{A_0}(0,t) \leq u^{A_0}(0, a_n)= u^{\bar A}(0, a_n)= u^{\bar A}(0,t).
$$
Hence $u^{A_0}$ is constant and equals $u^{\bar A}$ in $\{0\}\times [\bar s^-,a_n]$. This implies that $\bar A\equiv A_0$ in $(\bar s^-,a_n)$, a contradiction. 
\vs
Therefore $a_n= \bar s^-$.  This shows that $(\bar s^-,\bar a)=(a_n,b_n)$ is a connected component of $\{u^{\bar A}(0,\cdot)<u^{A_0}(0,\cdot)\}$.
\end{proof}

The main consequence of the previous lemma is the confinement of optimal trajectories vanishing in the last connected component.

\begin{cor} \label{cor.optistratlast} Assume \eqref{takis150}.  Let $\gamma$ be an optimal path for $u^{\bar A}(x,t)$ with   $(x,t)\in [x_1,x_2]\times [t_3,t_4]$.  If $|\{\gamma=0\}|>0$, then the last connected component $(\bar a,\bar b)$ of $\{u^{\bar A}<u^{A_0}\}$ exists and either $\{\gamma=0\}\subset [\bar a,\bar b]$ or $\{\gamma=0\}\subset [\hat a,\bar b]$ where $(\hat a,\bar a)$ is a one before the last connected component of $\{u^{\bar A}<u^{A_0}\}$. 
\end{cor}

\begin{proof} In view of  the structure of minimizers, there exists $s^-<s^+$ such that $\{\gamma=0\}=[s^-,s^+]$. Moreover, in view of  Lemma \ref{lem:condstrangeBIS},  we know that $s^+\geq \tau$. 
\vs
Fix $\ep\in (0,s^+-s^-)$. If  $\bar A \equiv A_0$ in $(s^+-\ep,s^+)$, then, by dynamic programming, if $\gamma_\ep$ is the straight-line connecting $(0, s^+-\ep)$ and $(x,t)$,
\begin{align*}
u^{\bar A}(x,t)& = \int_{s^+}^t L^R(\dot \gamma) ds -\int_{s^+-\ep}^{s^+} A_0ds +u^{\bar A}(s^+-\ep,0) \\
& =  \int_{s^+-\ep}^t L^R(\dot \gamma) ds  +u^{\bar A}(s^+-\ep,0) > \int_{s^+-\ep}^t L^R(\dot  \gamma_\ep) ds  +u^{\bar A}(s^+-\ep,0), 
\end{align*}
which is  a contradiction. 
\vs
So $\bar A\not \equiv A_0$ in $(s^+-\ep,s^+)$, and, hence, there exists $t_\ep\in  (s^+-\ep,s^+)$ such that $u^{\bar A}(0,t_\ep)<u^{A_0}(0,t_\ep)$. 
\vs
Let $(a, b)$ be the connected component of $\{u^{\bar A}<u^{A_0}\}$  containing $t_\ep$. We know from 
Proposition~\ref{prop.structopti} that $s^-\leq a$, which implies   that this connected component does not depend on $\ep$. In addition, $ b> t_\ep$, thus, as  $\ep\to0$,   $\bar b\geq s^+\geq \tau$. Thus $(a,b)$ must be the last connected component.
\vs
Assume now that $\{\gamma=0\}\not \subset [\bar a, \bar b]$ and let $\hat \gamma^+_{0,\bar b}$ be the optimal trajectory for $u^{\bar A}(0, \bar b)$ which remains the most in $\{x=0\}$. It follows from  Lemma \ref{lem.hatgammaphineg} that  $\{\gamma= 0\}\subset \{\hat \gamma^+_{0,\bar b}=0\}$ and thus $\{\hat \gamma^+_{0,\bar b}=0\}\not \subset  [\bar a, \bar b]$. Then Lemma \ref{lem.onebeforelast} states that, if we set $\bar s^-=\inf\{s, \; \hat \gamma^+_{0,\bar b}(s)=0\}$, then $(\bar s^-,\bar a)$ is the one before last connected component of $\{u^{\bar A}<u^{A_0}\}$. 
\end{proof}

\subsection{On the number of connected components}

We finally show in this subsection that, if \eqref{takis150} and \eqref{hyp:condu0} hold, then the number of  connected component of $\{u^{\bar A}<u^{A_0}\}$ is finite. 
\vs
Let $\tau$ be given by \eqref{deftautauBIS}. We define $\alpha\in [0,T]$ as follows: If there is no last connected component of $\{u^{\bar A}<u^{A_0}\}$, then  $\alpha=\tau$. If there is a last connected component $(\bar a, \bar b)$, but the component $(\hat a,\hat b)$  before last either does not exist or  satisfies $\hat b<\bar a$, then we set $\alpha=\bar a$. Finally, if  $\hat b=\bar a$, we set $\alpha =\hat a$. 

\vs

We first discuss a technical stability result. 
\begin{lem}\label{lem.stabilo2}
Given $\ep>0$, there exists $\delta>0$ such that, for any measurable $A:[0,T] \to [A_0,0]$ with  $\|A-\bar A\|_{L^1}\leq |A_0|\delta$,  any $(x,t)\in [x_1,x_2]\times [t_3,t_4]$ with  $u^{A}(x,t)<u^{A_0}(x,t)$ and for any optimal trajectory $\gamma$ for $u^A(x,t)$, 
$$
\inf\{ s : \; \gamma(s) =0\}\geq \alpha- \ep.
$$
\end{lem}

\begin{proof}  We first consider the case   $A\equiv \bar A$. Let $(x,t)\in [x_1,x_2]\times [t_3,t_4]$ and $\gamma$ be optimal for $u^{\bar A}(x,t)$. We show that $\inf\{ s:  \; \gamma(s) =0\}\geq \alpha$. 
\vs
If  $|\{\gamma=0\}|=0$,  then by the structure of minimizers and Lemma \ref{lem:condstrangeBIS} imply that 
$$
\inf\{s : \; \gamma(s)=0\}=\sup\{s : \; \gamma(s)=0\}\geq \tau\geq \alpha. 
$$
If  $|\{\gamma=0\}|>0$, then, in view of   Corollary  \ref{cor.optistratlast},  the last connected component $(\bar a,\bar b)$ of $\{u^{\bar A}<u^{A_0}\}$ exists and either
$$
\inf\{s : \; \gamma(s)=0\}\geq \bar a\geq \alpha, 
$$
or there exists a one before last connected component $(\hat a,\hat b)$ with $\hat b=\bar a$ and 
$$
\inf\{s : \; \gamma(s)=0\}\geq  \hat a\geq \alpha.
$$
\vs 
We now show that there exists $C>0$ such that, for any  $t\in  [\alpha, T]$,  any  optimal trajectory $\gamma$ of  $u^{\bar A}(0,t)$
satisfies, for all $s\in [0, s^-]$, the bound 
\be\label{laeiurdjtfg}
|\gamma(s)| \geq C^{-1}(s^--s)/s^-, 
\ee
where $s^-= \inf\{s :  \; \gamma(s)=0\}$. 
\vs
Indeed, let $\gamma$ be such a path. 
We know that $s^-\geq \alpha$. Moreover, by the structure of minimizers, $\gamma$ does not vanish on $[0, s^-)$ and $u^{\bar A}(0, s^-)= u^{A_0}(0, s^-)$. Thus, $\gamma_{|_{[0, s^-]}}$ is an optimal path for  $ u^{A_0}(0, s^-)$. Then  Lemma \ref{lem.hyp:optiuA0} and a compactness argument yield the existence  of some $C>0$, which is  independent of $\gamma$,  such that $|\gamma(0)|\geq C^{-1}$. Since  $\gamma_{|_{[0, s^-]}}$ is a straight line, \eqref{laeiurdjtfg} follows. 
\vs 

In view of  stability in $L^1$ of the paths $A$, for any $\ep>0$, there exists $\delta>0$ such that, for any measurable $A:[0,T] \to  [A_0,0]$ with $\|A-\bar A\|_{L^1}\leq \delta|A_0|$, for any $(x,t)\in [x_1,x_2]\times [t_3,t_4]$ and for any optimal trajectory $\gamma$ for $u^A(x,t)$, there exists an optimal trajectory $\tilde \gamma$ for $u^{\bar A}(x,t)$ such that $\|\gamma-\tilde \gamma\|_\infty\leq (CT)^{-1}\ep$, where $C$ is the constant in \eqref{laeiurdjtfg}. 
\vs
Let $s^-=\inf\{s: \; \tilde \gamma(s)=0\}$. We know that $s^-\geq \alpha$. Then, by \eqref{laeiurdjtfg}, 

$$
|\gamma(s)|\geq |\tilde \gamma(s)| -(CT)^{-1}\ep \geq C^{-1}(s^--s)/s^--(CT)^{-1}\ep.
$$

If $\gamma(s)=0$ for some $s\leq s^-$, then $s\geq s^--s^-T^{-1}\ep \geq \alpha- \ep$.  
\vs
The proof is now completed. 
\end{proof}

We show next a ``local finiteness'' result for the number of connected components of $\{u^{\bar A}<u^{A_0}\}\cap (0,\alpha-\ep)$. 

\begin{lem}\label{lem.locfinite} Assume \eqref{takis150} and let $\alpha$ be defined as above. Then, for any $\ep>0$, the set $\{u^{\bar A}<u^{A_0}\}\cap (0,\alpha-\ep)$ is the finite union  of disjoint intervals. 
Moreover, for any connected component $(a,b)$ with $b<\alpha$,  there exists $\delta>0$ such that $u^{A_0}(0, \cdot)=  u^{\bar A}(0, \cdot)=u^{A_0}(0,a)$ on $(a-\delta,a)$. 
\end{lem}

\begin{proof} 
Fix $\ep>0$ and let $\delta >0$ be as in Lemma \ref{lem.stabilo2}. If $(a,b)$ is a connected component of $\{u^{\bar A}<u^{A_0}\}$ such that $a\leq \alpha-\ep$, set
$$
A_\delta (s) = \left\{\begin{array}{ll}
0 &  {\rm if }\; s\in  (a-\delta, a),\\[1mm]
\bar A(s) & {\rm otherwise},
\end{array}\right.
$$
and observe that $\|A_\delta-\bar A\|_{L^1}\leq |A_0|\delta$ and $u^{A_\delta}\leq u^{\bar A}$. 
\vs
We claim that $u^{A_\delta}= u^{\bar A}$ in $[x_1,x_2]\times [t_3,t_4]$. Indeed, let $(x,t)\in [x_1,x_2]\times [t_3,t_4]$. If $u^{A_\delta}(x,t)=u^{A_0}(x,t)$, the claim is obvious. Otherwise, let $\gamma$ be optimal for $u^{A_\delta}(x,t)$. Then, by Lemma \ref{lem.stabilo2},  
$$
\inf\{s : \; \gamma(s)=0\} \geq \alpha-\ep \geq  a.
$$
Then  $\{\gamma=0\}\subset \{A_\delta =\bar A\}$, and, hence,  
$$
u^{A_\delta}(x,t)\leq u^{\bar A}(x,t)\leq J^{\bar A}(\gamma)= J^{A_\delta}(\gamma)= u^{A_\delta}(x,t),
$$
the first inequality being a consequence of the fact that $\bar A\leq A_\delta$, and, thus, 
$u^{A_\delta}= u^{\bar A}$ in $[x_1,x_2]\times [t_3,t_4]$. 
\vs
Since   $u^{A_\delta}\leq u^{\bar A}$ with an equality in $\{\phi<0\}$, we have that $\mathcal J(A_\delta)\leq \mathcal J(\bar A)$. 
\vs
We now show that $u^{ A_\delta }(0,a)=u^{\bar A}(0,a)$. Recalling  Lemma \ref{lem:variante52},  
%
for any point  $(x,t)\in \mathcal O\cap \{\phi> 0\}$ of differentiability of $u^{\bar A}(x,t)$, we set 
$$
s_{x,t}=\sup\{s :  \; \hat \gamma^-_{x,t}(s)=0\} \in (a,b). 
$$
Then, using the path  $\hat \gamma^-_{x,t}$ in the dynamic programming, we find 
$$
u^{A_\delta}(x,t)\leq \int_{s_{x,t}}^t L^R(\dot{\hat \gamma}^-_{x,t}) ds  -\int_a^{s_{x,t}} \bar A(s)ds +u^{ A_\delta }(0,a).
$$
If $u^{ A_\delta }(0,a)<u^{\bar A}(0,a)$, then 
$$
u^{A_\delta}(x,t)< \int_{s_{x,t}}^t L^R(\dot{\hat \gamma}^-_{x,t}) ds  -\int_a^{s_{x,t}} \bar A(s)ds +u^{\bar A }(0,a)= u^{\bar A}(x,t), 
$$
which implies, in view of the fact that  $|\mathcal O\cap \{\phi> 0\}|>0$, that $\mathcal J(A_\delta)<\mathcal J(\bar A)$, a contradiction with the optimality of $\bar A$. Thus $u^{ A_\delta }(0,a)=u^{\bar A}(0,a)$.
\vs
Using again  dynamic programming with   test trajectory $\gamma=0$,  for any $a-\delta< t<t'< a$, we have 
$$
u^{ A_\delta }(0,t') \leq -\int_t^{t'} A_\delta(s)ds +u^{A_\delta}(0,t)=u^{A_\delta}(0,t).
$$
Hence $u^{A_\delta}(0,\cdot)$ is nonincreasing on $(a-\delta,a)$. However, $u^{A_\delta}(0,\cdot)$ is nondecreasing because $u^{A_\delta}_t =-H^R(u^{A_\delta}_x)\geq 0$ a.e. in $(0,\infty)\times (0,T)$. Thus $u^{A_\delta}(0,\cdot)$ is constant on $(a-\delta,a)$. Then 
$$
u^{A_0}(a-\delta) \leq u^{A_0}(a)= u^{\bar A}(0,a)= u^{A_\delta}(0,a) = u^{A_\delta}(0, a-\delta), 
$$
so that $u^{A_0}(a-\delta) =u^{\bar A}(a-\delta)= u^{\bar A}(0,a)= u^{A_0}(0,a)$. Since  $u^{A_0}(0,\cdot)$ and $u^{\bar A}(0,\cdot)$ are both nondecreasing, this implies that they are constant and equal on $(a-\delta,a)$. 
\vs
In view of the fact that $u^{A_0}(0,\cdot)= u^{\bar A}(0,\cdot)$ in $[a-\delta, a]$, we have $\bar A\equiv A_0$ on $(a-\delta,a)$. If $(a',b')$ is another connected component of $\{u^{\bar A}<u^{A_0}\}$ such that $(a',b')\cap (a-\delta, a)\neq \emptyset$, then,  by  Lemma \ref{lem.cdBIS}, we necessarily have $a'<a-\delta$. 
\vs 
This shows that $(a-\delta, a)$ intersects  at most one connected component of $\{u^{\bar A}<u^{A_0}\}$, with $\delta$ independent of $(a,b)$. Therefore,  the number of connected components $(a,b)$ such that $a\leq \alpha-\ep$ is finite. 
\end{proof}

We now investigate further the last statement in Lemma \ref{lem.locfinite}. 

\begin{lem}\label{lem.flatpart} Assume \eqref{takis150} and let $\alpha$ be defined as above.  Let $(a,b)$ be a connected component of $\{u^{\bar A}(0,\cdot)<u^{A_0}(0,\cdot)\}$ such that $b<\alpha$, and  $\delta>0$ be such that $u^{A_0}(0, \cdot)= u^{\bar A}(0, \cdot)=u^{A_0}(0,a)$ on $(a-\delta,a)$. Then, for any $t\in (a-\delta,a)$, 
$$
{\rm either}\; u^{A_0}(0, t)= u_0(-t(H^R)'(-R^R)) \  {\rm or}\ u^{A_0}(0, t)= u_0(-t(H^L)'(0)).
$$
In addition, if there exists $t\in (a-\delta,a)$ such that $u^{A_0}(0, t)= u_0(-t(H^R)'(-R^R))$, then, for any other connected component $(a',b')$ of $\{u^{\bar A}(0,\cdot)<u^{A_0}(0,\cdot)\}$ with $b'\leq a$, we have  
$$
a-a'\geq ax_1/(-t (H^R)'(-R^R)).
$$ 
Otherwise, 
$$
u_{0,x}(-t(H^L)'(0)) =0 \; \text{for a.e.}\; t\in (a-\delta,a).
$$
\end{lem}

\begin{rmk}\label{rmk:casparticulier} Let 
$$
E^+=\{ t\in (a-\delta, a), \; u^{A_0}(0, t)= u_0(-t(H^R)'(-R^R))\}
$$
and 
$$
E^-= \{ t\in (a-\delta, a), \; u^{A_0}(0, t)= u_0(-t(H^L)'(0))\},
$$
and note that $(a-\delta, a)= E^+\cup E^-$. Then the proof below also shows that $u_{0,x}(-t(H^R)'(-R^R))= -R^R$ for a.e. $t\in E^+$ and $u_{0,x}(-t(H^L)'(0)) =0$ for a.e.  $t\in E^-$. Thus, if $0<u_{0,x}$ a.e. in $(-\infty,0)$ and $u_{0,x}>-R^R$ a.e. in $(0,\infty)$, the set $\{u^{\bar A}(0,\cdot)<u^{A_0}(0,\cdot)\}$ consists at most of the last connected component, together with the one before last if it touches the last one. 
\end{rmk}

\begin{proof} Let $\gamma$ be optimal for $u^{A_0}(0,t)$ for $t \in (a-\delta,a)$. We know that $\gamma$ is a straight line and that $y=\gamma(0)\neq 0$. Thus $\gamma(s)= y(t-s)/t$. 
\vs
For any $h$ small, $\gamma_h(s)= y(t+h-s)/(t+h)$ is a possible path for $u^{A_0}(0,t+h)$ and, thus, 
$$
u^{A_0}(0,t)=tL(-y/t)+u_0(y)= u^{A_0}(0,t+h)\leq (t+h) L(-y/(t+h))+u_0(y),
$$
where $L=L^R$ if $y>0$ and $L=L^L$ if $y<0$. As this holds for any $h$ small, we get
$$
L(-y/t)+ (y/t) L'(-y/t)=0.
$$
Recalling that $L(z)-z L'(z)=-H(L'(z))$, where $H=H^R$ if $y>0$ and $H=H^L$ if $y<0$, we infer that $H(L'(-y/t))=0$. Thus $L'(-y/t)= 0$ or $L'(-y/t)=-R$. In the former case, $-y/t= H'(0)>0$, so that $y<0$ and  $y=-t(H^L)'(0)$. In the  later  one, $-y/t= H'(-R)<0$, so that $y>0$ and $y=-t(H^R)'(-R^R)$. 
\vs
Assume now that there exists $\bar t\in (a-\delta,a)$ such that $u^{A_0}(0, \bar t)= u_0(-\bar t(H^R)'(-R^R))$, so that $ \gamma(s) = y(\bar t-s)/\bar t$ with $y= -\bar t(H^R)'(-R^R)$
is optimal for $u^{A_0}(0,\bar t)$. As $u^{\bar A}(0,\bar t)= u^{A_0}(0,\bar t)$ and $ \gamma$ vanishes only at $\bar t$, $ \gamma$ is also optimal for $u^{\bar A}(0,\bar t)$. 
\vs
Let $(a',b')$ be any other connected component of $\{u^{\bar A}<u^{A_0}\}$ such that $b'\leq a$. From Lemma \ref{lem:variante52}, we know that there exists $(x,t)\in (x_1,x_2)\times (t_1,t_2)$ point of differentiability of $u^{\bar A}(x,t)$ such that 
$$
s^+=\sup\{s : \; \hat \gamma_{x,t}^+(s)=0\} \in (a',b'). 
$$
The optimal trajectories $\gamma$ and $\hat \gamma^+_{x,t}$ cannot cross in $\{x>0\}\times (0,T)$. Since  $\gamma(s^+)>0$ while $\hat \gamma_{x,t}^+(s^+)=0$, it follows  that $\gamma - \gamma_{x,t}^+>0$ on $(s^+, \bar t\wedge t)$. Thus $t<\bar t$, since $ \gamma(\bar t)=0$ and $\hat \gamma^+_{x,t}$ is positive on $(s^+,t)$, and  $\gamma(t) \geq \gamma_{x,t}^+(t)=x$. Recalling the expression of $\gamma$ and as $t\geq s^+>a'$ and $\bar t < a$, one gets
$$
x_1\leq x  \leq  y(\bar t-t)/\bar t < y(a-a')/a, 
$$
which gives the claim. 
\vs
Assume now that $u^{A_0}(0, t)= u_0(-t(H^L)'(0))$ for any $t\in (a-\delta, a)$. Then for a.e.  $t\in (a-\delta, a)$ such that $u_0$ is differentiable  at $-t(H^L)'(0)$, $y= -t(H^L)'(0)$ is optimal in the cost $z\to -tL^L(-z/t)+u_0(z)$, so that $(L^L)'(-y/t)+u_{0,x}(y)=0$. This implies that $u_{0,x}(y)=-(L^L)'((H^L)'(0))=0$.
\end{proof}

We are now ready to prove that there is only finitely many connected components of $\{u^{\bar A}(0,\cdot)<u^{A_0}(0,\cdot)\}$. 

\begin{lem} \label{lem.finitenumberofCC} Assume \eqref{takis150} and \eqref{hyp:condu0}.  Then the set $\{u^{\bar A}(0,\cdot)<u^{A_0}(0,\cdot)\}$ consists in only finitely many intervals. 
\end{lem}

\begin{proof} Assume on the contrary that the number of connected components of $\{u^{\bar A}(0,\cdot)<u^{A_0}(0,\cdot)\}$ is infinite. 
\vs
In view of Lemma \ref{lem.locfinite},  it is possible to order  these connected components depending on  whether there is a last  component $(\bar a, \bar b)$ or not and whether, if there exists a last component, the component before the last touches $(\bar a, \bar b)$ at $\bar a$ or not. 
\vs
If there is no last component, we have 
$$
\{u^{\bar A}(0,\cdot)<u^{A_0}(0,\cdot)\}= \bigcup_{n\in \N} (a_n,b_n),
$$
where, for any $n$, $a_n<b_n\leq a_{n+1}$. If there  exists a  last connected component and  the previous component does not touch $(\bar a,\bar b)$, then 
$$
\{u^{\bar A}(0,\cdot)<u^{A_0}(0,\cdot)\}= \bigcup_{n\in \N} (a_n,b_n)\; \cup\; (\bar a, \bar b),
$$
where the $a_n$, $b_n$ are as above. 
Finally,  if the component before the last touches $(\bar a, \bar b)$ at $\bar a$, in which case we denote this component as $(a,\bar a)$, 
$$
\{u^{\bar A}(0,\cdot)<u^{A_0}(0,\cdot)\}= \bigcup_{n\in \N} (a_n,b_n)\; \cup\; ( a, \bar b)\cup (\bar a, \bar b), 
$$
where, again, $a_n$, $b_n$ are as above. \vs 
Let $\tilde a>0$ be the limit of the $a_n$'s and $b_n$'s and choose  $\delta_n>0$ so that, for any $n\in N$, $u^{\bar A}(0, \cdot)= u^{A_0}(0, \cdot)$ is constant on $[a_n-\delta_n, a_n]$. 
\vs

Assume that, along  a subsequence  $n_k$, there exist  $t_{n_k}\in (a_{n_k}-\delta_{n_k},a_{n_k})$ such that $u^{A_0}(0, t_{n_k})=u_0(-t_{n_k}(H^R)'(-R^R))$. Then it follows from  Lemma \ref{lem.flatpart} that, for any $k$ and any $k'< k$, $a_{n_k}-a_{n_{k'}}> a_{n_k} x_1/(-t_{n_k} (H^R)'(-R^R))$, a contradiction to the fact that the $a_n$'s  has a positive limit. 
\vs

It follows from  Lemma \ref{lem.flatpart}  that, for any $n$ large enough, say $n\geq n_0$,  and $\theta= (H^L)'(0)$,
$$
u_{0,x}(-t\theta) =0 \; \text{for a.e.}\; t\in (a_n-\delta,a_n). 
$$
Since, in view of  \eqref{hyp:condu0}, the set $\{x\leq 0:  \; u_{0,x}(x)=0\}$ consists a.e. only of finitely many intervals, there exists $n_1>n_2\geq n_0$ such that  $(-a_{n_1}\theta, -(a_{n_1}-\delta_{n_1})\theta)$ and $(-a_{n_2}\theta, -(a_{n_2}-\delta_{n_2})\theta)$ belong to the same interval of $\{x\leq 0 : \; u_{0,x}(x)=0\}$, which implies that  $u_{0}$ is constant on $[-a_{n_1}\theta,  -(a_{n_2}-\delta_{n_2})\theta]$. 
\vs
It then follows,  again by Lemma \ref{lem.flatpart},   that $u^{A_0}(0,a_{n_1})= u^{A_0}(0, a_{n_2})$, and, since $u^{A_0}(0,\cdot)$ is nondecreasing, we find  that $u^{A_0}(0,\cdot)$ is constant on $[a_{n_2}, a_{n_1}]$. In view of the fact that  $u^{\bar A}(0, x)=u^{A_0}(0,x)$ for $x= a_{n_1}$ and $x=a_{n_2}$, and, since $u^{\bar A}(0, \cdot)$ is nondecreasing, $u^{\bar A}(0, \cdot)$ must be constant and equal to $u^{A_0}(0, \cdot)$ on $[a_{n_2}, a_{n_1}]$. This contradicts the fact that $(a_{n_2}, b_{n_2})\subset (a_{n_2}, a_{n_1})$ is a connected component of $\{u^{\bar A}(0,\cdot)<u^{A_0}(0,\cdot)\}$. 
\vs
In conclusion, the number of connected components of  $\{u^{\bar A}(0,\cdot)<u^{A_0}(0,\cdot)\}$ must be finite. 

\end{proof}

We conclude the subsection with the remaining proof.

\begin{proof}[Proof of Theorem \ref{thm.mainexTOT}]  Lemma \ref{lem.finitenumberofCC} yields that  the set $\{u^{\bar A}(0,\cdot)<u^{A_0}(0,\cdot)\}$ consists of only finitely many intervals. Since  $\bar A\equiv A_0$ outside $\{u^{\bar A}(0,\cdot)<u^{A_0}(0,\cdot)\}$ and, for any connected component $(a,b)$, there exists $\tau\in (a,b)$ such that $\bar A\equiv 0$ in $(a,\tau)$ and $\bar A\equiv A_0$ in $(\tau,b)$,  the set $\{\bar A \equiv 0\}$ consists of  a finite number of disjoint  intervals. 

\end{proof}

\subsection{An example}

We present here  a class of examples for  which the optimal controls cannot be constant and equal to $0$ or to $A_0$. 
\vs
For simplicity of notation,  we assume here that $H^L=H^R$ and thus remove the exponents ``L'' and ``R'' in the notation. We recall that $H:[-R, 0]\to \R$ is strictly convex with $H(0)=H(-R)=0$ and $H$ has a minimum at $\hat p$.

\begin{prop} \label{prop.exAnottrivial} Assume in addition to \eqref{takis150}
that $u_0(x)=px$,  where $p\in (-R,\hat p)$. If 
\be\label{lejznrdgf}
\begin{split}
 &(i)\; x_2\leq H(p)t_1/p, \  
 (ii)\; t_2 H'(0)/ (H'(0)-H(p)/p)<t_3, \ \text{and} \\
&(iii)\; -t_2 A_0/(H(p)-A_0) < t_3,
\end{split}
\ee
then $A\equiv A_0$ and $A\equiv 0$ are not optimal. 
\end{prop}

We strongly use below that the traffic is congested in the entry line, that is, $p\in (-R,\hat p)$. 
As we show in the following proof, the first condition in  \eqref{lejznrdgf} implies that the control is active on the  whole set $\{\phi\neq 0\}$, which is a natural condition. Conditions (ii) and (iii) hold if $t_3$ is large enough. 

\begin{proof} Define the control 
$$A=\begin{cases}  0 \ \text{on} \  [0, t_2], \\
A_0 \ \text{in}  \ (t_2,T].\end{cases}$$
Recall that $u^{A}\leq u^{A_0}$ and $u^{A_0}(x,t)= px-H(p)t$. Then, using dynamic programming for the interval $(t_2, T)$, we find  
\begin{align*}
&u^A(x,t)=\\
&\left\{\begin{array}{ll}
\displaystyle \min\{-Rx {\bf 1}_{\{x\leq 0\}}, px-H(p)t\} \  \text{if}  \  (x,t)\in \R\times [0,t_2],\\[1.2mm]
\displaystyle  \min\{\min_{y\in[0, H(p)t_2/p]} (t-t_2)L((x-y)/(t-t_2))\ ,\  px-tH(p)\}   &  \ \text{if} \ (x,t)\in [0, \infty)\times [t_2,T].
\end{array}\right.
\end{align*}

We first claim that $\mathcal J(A)< \mathcal J(A_0)$. Note that 
\be\label{aiuzsldnfcv}
\begin{split} 
 \mathcal J(A_0)- \mathcal J(A) & = \int_0^{t_2} \int_0^{H(p)t/p} \phi(x,t)(px-H(p)t) dxdt \\& \qquad + \int_{t_2}^T\int_{\R_+} \phi(x,t)(u^{A_0}(x,t)-u^{A}(x,t))dxdt. 
 \end{split}
 \ee
 Since, by \eqref{lejznrdgf}(i),  $x_2\leq H(p)t_1/p$, it follows that 
 \begin{align*}
 (x_1,x_2)\times (t_1,t_2) = \{\phi>0\} & \subset \{(x,t)\in \R_+\times [0,t_2], \; x< H(p)t/p\} \\ &\qquad  =\{(x,t)\in \R_+\times [0,t_2], \; u^{A}(x,t)< u^{A_0}(x,t)\},
\end{align*}
and, thus,  the first integral in \eqref{aiuzsldnfcv} is positive. 
\vs

We now fix $(x,t)$ such that $t> t_2$ and $u^{A}(x,t)< u^{A_0}(x,t)$. Then,  
$$
u^{A}(x,t)=\min_{y\in[0, H(p)t_2/p]} (t-t_2)L((x-y)/(t-t_2))<px-tH(p). 
$$
Let $y$ be a minimum  point $\min_{y\in[0, H(p)t_2/p]} (t-t_2)L((x-y)/(t-t_2))$. In view of the facts  that  $u^{A}(x,t)< u^{A_0}(x,t)$ and  $u^{A}(H(p)t_2/p, t_2)= u^{A_0}(H(p)t_2/p, t_2)$, it follows  that $y\in [0, H(p)t_2/p)$ . 
\vs

In the case where the minimum  point $y$ is in $(0, H(p)t_2/p)$, we have $L'((x-y)/(t-t_2))=0$, and thus 
$
\min_{y\geq0} (t-t_2)L(x-y/(t-t_2))= 0$ and $ x-y=(t-t_2)H'(0).$
Then $0=u^{A}(x,t)$ and $u^{A}(x,t)< u^{A_0}(x,t)$ implies that 
$0\leq x < H(p)t/p \ \text{and} \ x>(t-t_2)H'(0).$ It follows that 
$H'(0)(t-t_2)< H(p)t/p$, which in turn yields  that $t< t_2 H'(0)/ (H'(0)-H(p)/p))$, and, in view of  \eqref{lejznrdgf}-(ii),   
that $t<t_3$. 
\vs
If  the minimum point $y$ is zero, then 
$u^{A}(x,t)< u^{A_0}(x,t)$ implies, using the  convexity of $L$, 
$
px-H(p)t > (t-t_2)L(x/(t-t_2))\geq  \hat px -(t-t_2) H(\hat p)=\hat px -(t-t_2) A_0,
$
so that 
$
(\hat p-p) x+(H(p)-A_0)t < -t_2 A_0. 
$
Since, in view of  \eqref{lejznrdgf}-(iii), we have $t_3 > -t_2 A_0/(H(p)-A_0)$ and, hence,   $t\leq t_3$. 
\vs
The considerations above  show that  $\{u^{A}<u^{A_0}\}\cap (\R_+\times[t_2,T])\subset \{\phi\geq 0\}$, and,  hence, the second integral in \eqref{aiuzsldnfcv} is nonnegative. Thus $\mathcal J(A_0)> \mathcal J(A)$ and the constant control $A_0$ is not optimal. \vs

We check next that the constant control $0$ is not optimal. Note that 
$$
u^0(x,t) = \min\{-Rx {\bf 1}_{\{x\leq 0\}}, px-H(p)t\}\qquad \text{for } (x,t)\in \R\times [0,T]. 
$$
Hence, since $x_2\leq H(p)t_1/p$,  
$$
\{\phi\neq 0\} \subset [x_1,x_2]\times [t_1,t_4]\subset \{(x,t)\in \R_+\times [0,t_2], \; x\leq H(p)t/p\} \subset  \{u^0(x,t) =0\}.
$$

One the other hand we have seen that $u^{A}=u^0$ in $\R\times [0,t_2]$, while, in $(x_1,x_2)\times [t_3,t_4]$, $u^A= u^{A_0}>u^0$. This proves that 
$\mathcal J(0)>\mathcal J(A)$ and, thus,  the constant control $0$ is not optimal. 

\end{proof}

\end{document}